\documentclass[11pt,reqno]{amsart}
\usepackage{amssymb, amsmath, amsfonts, amsthm, graphics}
\usepackage[hmargin=1 in, vmargin = 1 in]{geometry}
\usepackage{mathrsfs}
\usepackage{newpxtext}
\usepackage{newpxmath}
\usepackage{tikz-cd} 
\usepackage{enumitem}
\setlist{itemsep=0.5em}
\usetikzlibrary{matrix, calc, arrows} 
\usepackage[hyphens]{url}
\usepackage{breakurl}

\usepackage{hyperref}
\usepackage[all]{xy}
\usepackage{marginnote}
\usepackage[backend=bibtex,
style=alphabetic,maxnames=5]{biblatex}
\AtBeginBibliography{\small}

\newcommand{\isom}{\cong} 

\theoremstyle{definition}

\numberwithin{equation}{section}

\DeclareMathOperator{\Spec}{\textbf{Spec}}

\DeclareMathOperator\codim{codim}

\newcommand{\vol}{\mathrm{vol}}

\renewcommand{\AA}{\mathbb{A}}

\newcommand{\NN}{\mathbb{N}}

\newcommand{\QQ}{\mathbb{Q}}
\newcommand{\RR}{\mathbb{R}}

\newcommand{\ZZ}{\mathbb{Z}}

\newcommand{\cF}{\mathcal{F}}

\newcommand{\cH}{\mathcal{H}}

\newcommand{\cL}{\mathcal{L}}

\newcommand{\cO}{\mathcal{O}}

\newcommand{\fm}{\mathfrak{m}}
\newcommand{\fa}{\mathfrak{a}}

\newcommand{\iI}{\mathscr{J}}
\newcommand{\iL}{\mathscr{L}}
\newcommand{\iF}{\mathscr{F}}

\theoremstyle{plain}
\newtheorem{thm}{Theorem}[section]
\newtheorem{Pn}[thm]{Proposition}
\newtheorem{Cor}[thm]{Corollary}
\newtheorem{lem}[thm]{Lemma}

\theoremstyle{definition}
\newtheorem{dfn}[thm]{Definition}

\newtheorem{eg}[thm]{Example}
\newtheorem{rem}[thm]{Remark}
\newtheorem{ques}[thm]{Question}
\newtheorem{notation}[thm]{Notation}

\begin{document}
	
	\title{Multiplicities of jumping numbers}
	\author{Suchitra Pande}
	\address[S.~Pande]{Department of Mathematics\\University of Michigan\\Ann Arbor, 
		MI 48109-1043\\USA}
	\email{\href{mailto:swarajsp@umich.edu}{swarajsp@umich.edu}}
	\thanks{This project was partially supported by NSF Grants DMS 1801697 and 2101075.}
	
	\begin{abstract}
		We study multiplicities of jumping numbers of multiplier ideals in a smooth variety of arbitrary dimension. We prove that the multiplicity function is a quasi-polynomial, hence proving that the Poincar\'e series is a rational function. We further study when the various components of the quasi-polynomial have the highest possible degree and relate it to jumping numbers contributed by Rees valuations. Finally, we study the special case of monomial ideals.
	\end{abstract}
	
	\maketitle

	\section{Introduction}
	Let $X$ be a smooth variety over an algebraically closed field of characteristic $0$. Associated to each closed subscheme $Z$ of $X$ is a family of ideals, called the {\it multiplier ideals} (Definition~\ref{MultiplierIdeal}),  that quantify  the singularities of $Z$. The multiplier ideals are indexed by a positive real parameter $c$ (denoted by $\iI(c \cdot Z)$), and form a decreasing family of ideals of $\cO_{X}$.
	
	As $c$ varies over the positive real numbers, the stalks of these ideals $\iI( c \cdot Z)_{x}$ at any point $x$ change exactly at a discrete set of rational numbers $c_{i}$ called the \emph{jumping numbers} of $Z$ at $x$ (Definition~\ref{jumpingnumbers}). So we get a descending chain of ideals
	\begin{equation} \label{chain}
		\cO_{X, x} \supsetneqq \iI(c_{1} \cdot Z)_{x} \supsetneqq \iI(c_{2} \cdot Z)_{x}  \supsetneqq \cdots	.	
	\end{equation}\
	The jumping numbers, first  defined and studied in \cite{MR2068967}, are interesting invariants of the singularity of $Z$ at $x$. For example, the smallest jumping number is the well-known log canonical threshold ($lct$) of the subscheme. The log canonical threshold, and more generally any  jumping number in the interval $[lct, lct + 1 ) $ is a  root of $b_{Z}(-s),$ where $b_{Z}(s)$ is the famous \emph{Bernstein-Sato polynomial} of $Z$; see \cite{MR1492525}, \cite{MR2068967} and \cite{MR2231202}. Jumping numbers were connected to the Hodge Spectrum of a hypersurface by Budur in \cite{MR2015069}.

	 Jumping numbers have been studied extensively in the case when dimension of $X$ is two, starting with \cite{MR2389246}. For example, explicit formulas for the jumping numbers have been calculated in \cite{MR1704476}, \cite{MR3479548}, \cite{MR2856648}, \cite{MR2470185} and \cite{MR3863487}. More algorithms for computing jumping numbers can be found in \cite{MR2592954}, \cite{MR3510908} and \cite{MR3558221}.
	
The purpose of this paper is to study jumping numbers in higher dimensions. We do this by studying a natural refinement of the jumping numbers, namely \emph{multiplicities of jumping numbers}, first introduced in \cite{MR2068967}. More precisely, fix a closed subscheme $Z$ of a smooth variety $X$, and an irreducible component $Z_{1}$ of $Z$. 
	All the (stalks of) multiplier ideals of $Z$ will have finite colength in the local ring at $Z_{1}$. 
	So any jumping number $c$ at $Z_{1}$ has a natural \emph{multiplicity $m(c)$} that measures the change in the multiplier ideal at $c$, namely
	
	\begin{equation}\label{eq} 
	m(c) :=	\lambda( \iI(\fa^{c -\varepsilon})_{x}/\iI(\fa^{c})_{x})   \quad \text{for $0 < \varepsilon \ll 1$.}    
	\end{equation}
	 Here,  $\iI(\fa^{c})_{x} $ denotes the stalk of the multiplier ideal of $Z$ at $x$, the generic point of $Z_{1}$ in $X$ and $\lambda$ denotes the length as an $\cO_{X,x}$ module. If the real number $c$ is not a jumping number, we define its multiplicity to be zero, compatibly with (\ref{eq}).

	
	 For any jumping number $c$, we obtain another jumping number by adding any positive integer, so it is natural to consider the sequence of multiplicities of the jumping numbers $c+n$, as $n$ ranges through the natural numbers. Our first main theorem is

\theoremstyle{plain}	
\newtheorem*{thm1intro}{Theorem \ref{thm3.1}}

\begin{thm1intro}
Let $Z$ be a closed subscheme of $X$ and $Z_{1}$ an irreducible component  of $Z$. For any positive real number $c$ and natural number $n$, let $m(c+n)$ denote the multiplicity of $c+n$ of $Z$ along $Z_1$. Then the sequence of multiplicities $$(m(c+n))_{n\in \mathbb N}$$ is a polynomial function of $n$ of degree less than the codimension of $Z_{1}$ in $X$.
\end{thm1intro}

Theorem~\ref{thm3.1} allows us to interpret the multiplicity function $m(c)$ as a quasi-polynomial in $c$; see Corollary \ref{quasipoly}. 

Theorem~\ref{thm3.1} generalizes the work of Alberich-Carrami\~{n}ana, \`Alvarez Montaner, Dachs-Cadefau and Gonz\'{a}lez-Alonso when $X$ is a surface \cite{MR3558221}. They compute the multiplicities explicitly in terms of the intersection matrix of the exceptional divisors in a log resolution. In higher dimension, we instead use the numerical intersection theory of divisors as developed by Kleiman in \cite{MR206009}.

The polynomial of Theorem~\ref{thm3.1} encodes interesting information about the divisors relevant for computing jumping numbers as developed in \cite{MR2389246}. For example, its degree tells  us  precisely whether (possibly some translate of) the jumping number is contributed by a Rees valuation:

\theoremstyle{plain}	
\newtheorem*{thm2intro}{Theorem \ref{contrithm}}

	\begin{thm2intro}
		For a closed subscheme $Z$ of $X$ and any positive real number $c$, consider the multiplicity polynomial $m(c+n)$ along an irreducible component $Z_1$  of codimension $h$ in $X$.
		Then the  degree of this polynomial is equal to $h-1$ if and only if  $c+h-1$ is a jumping number contributed---in the sense of  \cite{MR2389246}---by some Rees valuation of $Z$ centered at $Z_{1}$.
	\end{thm2intro}
	
	Theorem~\ref{contrithm} motivates the term \emph{Rees coefficient} of the jumping number $c$ for the coefficient of $n^{h-1}$ in the polynomial $m(c+n)$ from Theorem~\ref{thm3.1}. The Rees coefficient is not necessarily the ``leading coefficient" of the polynomial $m(c+n)$ because it can be zero; see Example~\ref{example}. 
	
	We prove several applications of 	Theorem~\ref{contrithm}. For example, we show that certain jumping numbers of $Z$ can be computed directly from the normalized blowup of $Z$ without finding a full log resolution; see Corollary~\ref{maincor} for a precise statement. As another consequence, we answer a question posed by Joaqu\'{i}n Moraga: we prove that every divisorial valuation over $X$ contributes a jumping number for some divisor; see Theorem~\ref{valuationthm}. 
	
  The Rees coefficient of a jumping number $c$ is the same as the Rees coefficient of any translate of $c$ by an integer. So we can consider the Rees coefficient of a class of jumping numbers modulo $\ZZ$; there are finitely many such classes, by discreteness of jumping numbers. We prove in Theorem~\ref{HSmult} that the sum of the Rees coefficients for the distinct classes of jumping numbers modulo $\mathbb Z$ 
  is equal to the
  \emph{Hilbert-Samuel multiplicity} of $Z$ at $Z_{1}$  (scaled by $\frac{1}{(h-1)!}$). So the Rees coefficients can be thought of as refinements of the Hilbert-Samuel multiplicity.
	
	 In Section 5,  we study the special case of point schemes defined by \emph{monomial ideals}. We prove formulas for the multiplicities and for the Rees coefficients of each jumping number in this case (Theorem~\ref{monomial}). In particular, we see that the Rees coefficient of every jumping number of  a monomial  scheme is positive (Corollary~\ref{anotcor}). Thus Theorem~\ref{contrithm} implies that, for monomial ideals, all jumping numbers (after translation by some integer)  are contributed by Rees valuations.
	
	Finally, in Section 6, we examine a generating function for multiplicities of a jumping number. For an irreducible component $Z_1$ of a fixed subscheme $Z$, we define a Poincar\'e series from the multiplicities $m(c)$, and prove that it is a rational function in a suitable sense; see Theorem \ref{poincare}. This generalizes the previous results from \cite{MR2671187} and \cite{MR3558221}, valid for  point schemes in dimension two. Theorem \ref{poincare}, which is valid in any dimension,  was independently proved by \`{A}lvarez Montaner and N\'{u}\~{n}ez-Betancourt in \cite{AMNB21} using different methods.

	\subsection*{Acknowledgements}
	I would first like to thank my PhD thesis advisor Prof. Karen Smith for generously sharing her time and ideas, and closely guiding me throughout this project. I thank Josep \`{A}lvarez Montaner and Luis N\'{u}\~{n}ez-Betancourt for communicating their results and kindly sharing their preprint \cite{AMNB21} with me. I also thank them for suggesting a useful modification to the statement of Theorem~\ref{poincare}. I thank the referee for many useful suggestions for improving the paper. I thank Joaqu\'{i}n Moraga for raising questions that led to Theorem~\ref{valuationthm}. I thank Prof. Suresh Nayak for comments on an earlier draft. Finally I would like to thank Shelby Cox, Andrew Gordon, Sayantan Khan, Devlin Mallory, Malavika Mukundan, Sridhar Venkatesh and Yinan Nancy Wang for helpful conversations.
	
	\section{Review of Multiplier Ideals and Intersection Theory}

Throughout this paper, we work over an an algebraically closed field $k$ of characteristic zero.

	\subsection{Multiplier Ideals and Jumping Numbers}
	\label{multiplier}
	Let $X$ be a smooth variety over $k$. Fix a coherent ideal sheaf $\fa$ of $\cO_{X}$ and let  $Z$ be the subscheme defined by $\fa$. We will now recall the  definition of the multiplier ideals $ \iI(c \cdot Z)$, interchangeably denoted by  $\iI(\fa^{c}) $,  referring the reader to \cite{MR2095472} for details. 
	
	\begin{dfn} \label{logresolution} A {\it log resolution } of the ideal $\fa$ is a map $\mu: Y \to X $ that is projective and birational such that:
	\begin{itemize}
	    \item [(a)] $Y$ is smooth,
	    \item [(b)] $\fa \cdot \cO_{Y} = \cO_{Y}(-F)$ where $F$ is an effective divisor and
	    \item [(c)] $ F+K_{Y|X}$ has simple normal crossing support,
	where   $ K_{Y|X} $  denotes the relative canonical divisor of $\mu$.
	\end{itemize}
	\end{dfn}
	
     Such log resolutions exist by Hironaka's theorem on resolution of singularities.
	
	\begin{dfn} \label{MultiplierIdeal}
	    For any positive real number $c$,  the \emph{multiplier ideal} of $\fa$ at $c$ is defined as $$\iI( \fa^{c}) = \mu _{*} (\cO_{Y}(K_{Y|X} - \lfloor cF \rfloor)),$$
	    where $\mu: Y \to X$  is any log resolution of $\fa$ and $F$ is an effective divisor with $\fa \cdot \cO_{Y} = \cO_{Y}(-F)$. The multiplier ideal is independent of the choice of the log resolution.
	\end{dfn}

 As the positive real parameter $c$ increases, the stalks of the multiplier ideal $\iI( \fa^{c})_{x}$ at any point $x$ decrease. The numbers $c$ at which the stalks change are called the \emph{jumping numbers} of the ideal $\fa$ (or $Z$) at $x$. More precisely,
	 
	 \begin{dfn} \cite{MR2068967} \label{jumpingnumbers}
	     A positive real number $c$ is a \emph{jumping number} of the subscheme $Z$ at a point $x \in X$ if
	     $$ \iI(\fa^{c})_{x} \subsetneqq \iI(\fa^{c - \varepsilon})_{x} \quad \text{for all $\varepsilon > 0 $.}     $$
	 \end{dfn}
	 It follows from the definition of the multiplier ideal in terms of a log resolution that the jumping numbers of any subscheme $Z$ are discrete and rational (see Section~\ref{candidate} below). 
	 
	
	\subsubsection{\textbf{Candidate Jumping Numbers:}} \label{candidate} Let $\mu: Y \to X$  be any log resolution of $\fa$ with $\fa \cdot \cO_{Y} = \cO_{Y}(-F)$ for an effective divisor $F$. Suppose $F = \sum _{i} a_{i} D_{i} $ for prime divisors $D_{i}$. Then, we call the numbers of the form $\frac{n}{a_{i}}$ for natural numbers $n$ the \emph{candidate jumping numbers} of $\fa$. These are precisely the values of $c$ where the divisor $K_{Y|X} - \lfloor cF \rfloor$ changes, hence the set of jumping numbers at any point is certainly contained in the set of candidate jumping numbers. The candidate jumping numbers make it clear that the set of jumping numbers is rational and discrete. Moreover, there is a uniform $\varepsilon$ such that $c_{i+1} -c_{i} > \varepsilon$ for any two consecutive jumping numbers $c_{i}$ and $c_{i+1}$.
	
	Note that this definition is slightly different from the candidate jumping numbers as defined in \cite{MR2389246} where they were defined to be the set of rational numbers where the divisor  $K_{Y|X} - \lfloor cF \rfloor$ changes and is not effective. The main reason for this deviation is that now for every candidate jumping number $c$, its fractional part $\{c\}$ is also a candidate jumping number. Hence, every jumping number $c$ can be written as $\{c\} + \lfloor c \rfloor$ i.e., an integer translate of a candidate jumping number that lies in the interval $(0,1]$. \\

	We now recall two of the main results from the theory of multiplier ideals in the form that we will use them. The reference for these results is \cite[Chapter 9]{MR2095472}.
	
	\begin{thm}{\textbf{(Local Vanishing Theorem)}} \label{localvanishing}
		
		Let $\fa \subset \cO_{X}$ be an ideal, and $\mu: Y \to X$ be a log resolution of $\fa$ and let $\fa \cdot \cO_{Y} = \cO_{Y}(-F)$ for an effective divisor $F$. Then for any $c >0$, the following higher direct images vanish:
		$$ R^{j} \mu_{*} \cO_{Y}(K_{Y|X} - \lfloor cF\rfloor ) = 0 \quad	\mathrm{	for } \quad j >0. $$
	\end{thm}

	\begin{thm} {\textbf{(Skoda's Theorem)}} \label{skoda}
		
		Let $\fa \subset \cO_{X}$ be an ideal and $c>0$ be any real number and $m \geq d := \dim(X)$ be an integer. Then we have,
		$$ \iI(\fa^{c+m}) = \fa \cdot \iI(\fa^{c+m-1}).	$$
	\end{thm}

	\subsection{Numerical Intersection Theory}
	
	Now we review Kleiman's numerical intersection theory of divisors as developed in \cite{MR206009} and also explained in \cite{MR1841091}. We only state the main facts that we need here.
	
	Let $Y$ be a proper scheme of dimension $d$ over a field $L$. Let $\iL_{1}, \dots, \iL_{r}$ be line bundles on $Y$ and $\iF$ a coherent sheaf on $Y$. Then we have the following theorem from \cite{MR206009} that Kleiman attributes to Snapper:
	
	\begin{thm} \label{Snapper}
		Consider the function $f(m_{1}, \dots, m_{r}) = \chi (\iF \otimes \iL_{1} ^{m_{1}} \otimes  \dots \otimes \iL_{r}^{m_{r}})$ for $m_{1}, \dots, m_{r} \in \ZZ$, where $\chi$ denotes the Euler characteristic (over $L$). Then there is a polynomial $P(x_{1}, \dots, x_{r})$ with coefficients in $\QQ$ and of total degree $\leq \dim \mathrm{Supp}(\iF)$ such that $P(m_{1}, \dots, m_{r}) = f(m_{1}, \dots, m_{r})$ for all $m_{1}, \dots, m_{r} \in \ZZ$.
	\end{thm}
	
	\begin{dfn} \cite[Section 2, Definition 1]{MR206009}
		Suppose that $\dim \mathrm{Supp}(\iF) \leq r $, then define the intersection number $ 	(\iF; \iL_{1} \cdot \ldots \cdot \iL_{r})$ to be the coefficient of the term $x_{1} \cdots x_{r}$ in the polynomial $P(x_{1}, \dots, x_{r})$ as in Theorem~\ref{Snapper}. In particular, if $\dim \mathrm{Supp}(\iF) < r$ then $ 	(\iF; \iL_{1} \cdot \ldots \cdot \iL_{r}) = 0$, since the degree of $P(x_{1}, \dots, x_{r})$ is strictly less than $r$ by Theorem~\ref{Snapper}.
	\end{dfn}
	
	For any coherent sheaf $\iF$ with $\dim (\mathrm{Supp}(\iF)) \leq r$, this defines an integer valued multi-linear form on Pic$(Y)^{r}$. If a line bundle $\iL$ is defined by a Cartier divisor $D$, then we sometimes write $D$ instead of $\iL$ in the intersection form. We will use the following properties of the intersection numbers:
	
	\begin{Pn} \label{interpn} Let $Y$ be a proper scheme over a field $L$ of dimension $d$. Let $\iL_{1}, \dots, \iL_{r}$ be line bundles on $Y$, $D$ an effective Cartier divisor and $\iF$ a coherent sheaf on $Y$. Suppose $\dim \mathrm{Supp}(\iF) = r$. Then:
		\begin{enumerate}
			
			\item \label{intersection1} If $\iF$ is a locally free sheaf (in this case $r=d$), then $$(\iF; \iL_{1} \cdot \ldots \cdot \iL_{d-1} \cdot \cO_{Y}(D )) =  (\iF|_{D} ; \iL_{1}|_{D} \cdot \ldots \cdot \iL_{d-1}|_{D}) .$$

			\item	\label{intersection2}	 If $\iF^{\prime} $ and $\iF^{\prime \prime} $ are two other coherent sheaves and there is an exact sequence
			$$ 0 \to \iF^{\prime} \to \iF \to \iF^{\prime \prime} \to 0 ,  $$
			then $(\iF; \iL_{1} \cdot \ldots \cdot \iL_{r}) = (\iF^{\prime}; \iL_{1} \cdot \ldots \cdot \iL_{r}) + (\iF^{\prime \prime}; \iL_{1} \cdot \ldots \cdot \iL_{r}).$

			\item \label{intersection3} If $\iL$ is any line bundle on $Y$ and $P(x)$ is the polynomial such that $P(n) = \chi(\iF \otimes \iL^{n})$ for $n \in \ZZ$ (which exists by Theorem~\ref{Snapper} above), then $$ (\iF; \underbrace{\iL\cdot \ldots \cdot \iL}_{r \text{ times}})= \alpha \ r! $$
			where $\alpha$ is the coefficient of $x^{r}$ in $P(x)$.
			
			\item \label{intersection4} Let $V = \mathrm{Supp}(\iF)$ and $V_{1}, \dots, V_{s}$ its irreducible components. Let $l_{i} = \mathrm{length}(\iF \otimes \cO_{V_{i}} )$ where $\cO_{V_{i}}$ is the stalk of $\cO_{Y}$ at the generic point of $V_{i}$, then assuming $r \geq \dim(V)$, we have
			$$ (\iF;\iL_{1} \cdot \ldots \cdot \iL_{r}) = \sum _{i=1} ^{s} l_{i} \ (\iL_{1}|_{V_{i}} \cdot \ldots \cdot \iL_{r}|_{V_{i}})	.	$$
			In particular, if $\iF$ is an invertible sheaf, then $(\iF; \iL_{1} \cdot \ldots \cdot \iL_{d}) = (\iL_{1} \cdot \ldots \cdot \iL_{d})$. 
			
			\item  \label{intersection5} Let $\pi: Y^{\prime} \to Y$ be a map of finite type of integral projective varieties of dimension $d$, then $ (\pi^{*} \iL_{1} \cdot \ldots \cdot \pi^{*} \iL_{d})_{Y^{\prime}} =  \text{deg} (\pi) \ (\iL_{1} \cdot \ldots \cdot \iL_{d})_{Y}$.
		\end{enumerate}
		
	\end{Pn}
	
	\begin{proof} Parts (1), (2), (4) and (5) are proved in \cite{MR206009}, Section 2, Propositions 4, 3, 5 and 6 respectively. Note that even though it is assumed that the ground field $L$ is algebraically closed in \cite{MR206009}, the hypothesis is not required in Section 2. For instance, see \cite{MR1841091}.
	
	(3): If $Q(x_{1}, \dots, x_{r})$ is the polynomial such that $Q(m_{1}, \dots, m_{r}) = \chi(\iF \otimes \iL^{m_{1}} \otimes \ldots \otimes \iL^{m_{r}})$, then $Q(x_{1}, \dots, x_{r}) = P(x_{1} + \dots + x_{r})$. Since $P(x)$ has a degree at most $r$, the coefficient of $x_{1} \dots x_{r}$ in $P(x_{1} + \dots + x_{r})$ is $\alpha \, r!$ which by definition is the required intersection number.
	\end{proof}
	
	
	\section{The Polynomial Nature of Multiplicities}{\label{Polynomial}}
	In this section, we prove Theorem~\ref{thm3.1}  on the polynomial behavior of multiplicities of jumping numbers. We begin with some preliminary definitions.
	
	\begin{notation}
	We fix the following notation throughout this section: Let $X$ be a smooth variety of dimension $d$ over $k$. We fix a closed subscheme $Z$ of $X$ and an irreducible component $Z_{1}$ of $Z$. Let $x$ be the generic point of $Z_{1}$ in $X$ and $(A, \fm, L)$ the local ring at $x$ in $X$. The Krull dimension of $A$--- or equivalently, the codimension of $Z_1$ in $X$--- will be denoted $h$. Let $\fa \subset \cO_{X}$ denote the ideal of $Z$, and observe that the stalk of $\fa$ at $x$ is an $\fm-$primary ideal of $A$. Abusing notation, we often write $\fa$ and  $\iI(\fa^{c})$ for the stalks of the ideal sheaf $\fa$ and the multiplier ideal at $x$ respectively and think of these as ideals in $A$.
	\end{notation}
	
	\begin{dfn} \cite{MR2068967} \label{multidef}
		For any real number $c>0$ of $\mathfrak a$ at $x$, the \emph{multiplicity of $c$ at $x$} is defined to be
		\begin{equation}\label{multdef} m(c) :=	\lambda( \iI(\fa^{c -\varepsilon})_{x}/\iI(\fa^{c})_{x}),	
		\end{equation}
		for sufficiently small positive $\varepsilon.$   Here, $\lambda$ denotes the length as an $A-$module.
	\end{dfn}
	
	The multiplicity is well-defined because $\iI(\fa^{c})$ is either $A$ or an $\fm-$primary ideal for all $c$ since $\fa$ is $\fm-$primary and $\iI(\fa^{c}) \supset \iI(\fa^{\lceil c \rceil }) \supset \fa^{\lceil c \rceil}$. Note that even though we define $m(c)$ for any positive real number $c$ by expression (\ref{multdef}), it will be non-zero if and only if $c$ is a jumping number, by the definition of jumping number. 
	
	Our focus in this section is on the function $m(c+n)$ where $c>0$ is any fixed real number and $n$ varies over the natural numbers $\ZZ_{\geq 0}$. So, we define the function
	$$ f_{c}(n) := m(c+n)	.		$$
	
	The main theorem about $f_{c}(n)$ is the following:
	
	\begin{thm} \label{thm3.1}
		Let $Z$ be a closed subscheme of a smooth variety $X$ and $Z_{1}$ be an irreducible component of $Z$. Then for each $c >0$, the function $f_{c}(n)$ is a polynomial function of $n$ of degree less than the codimension $h$ of $Z_{1}$ in $X$.
	\end{thm}
	
	Recall that a function $g: \RR_{> 0} \to \RR$ is called a \emph{quasi-polynomial} if $g$ can be written as
	$$ g(x) = a_{r}(x) x^{r} + \dots + a_{0}(x),       $$
	where each $a_{i}(x)$ is a periodic function of $x$ with integral period.
	
	\begin{Cor}\label{quasipoly}
	
	Let $Z$ be a closed subscheme of a smooth variety $X$. The multiplicity function $m(c)$ of $Z$ along one of its components $Z_{1}$ is a quasi-polynomial in $c$.
	
	\end{Cor}
	
	\begin{proof}
	    By Theorem~\ref{thm3.1}, for each $c$ in the interval $(0,1]$, the multiplicity $m(c+n)$ can be written as a polynomial $P_{c}(c+n)$ of degree less than $h$, where $h$ is the codimension of $Z_{1}$ in $X$. So we can write $m(c)$ as
	    $$ m(x) = a_{h-1}(x) x^{h-1} + \dots + a_{0}(x)     $$
	   where $a_{i}(x)$ is  the coefficient of $x^{i}$ in the polynomial $P_{c}$, for $c$ the fractional part $c = x-\lfloor x\rfloor$ of $x$. Note that each $a_i$ is a periodic function with period 1.
	\end{proof}

	Before we prove the theorem, it will be convenient to introduce some notation and the notion of \emph{jumping divisor} of a candidate jumping number:
	\begin{dfn} \label{jumpingdiv} Let $c$ be  positive real number $c$ and  $\fa$ an  ideal sheaf on the smooth variety $X$.
	Fix a log resolution $\mu: Y \to X$ of $\fa$, and let $F$ be the unique effective exceptional divisor such that $\fa \cdot \cO_{Y} = \cO_{Y}(-F)$. 
		For $\varepsilon > 0$ small enough, the divisor $ \lfloor cF \rfloor - \lfloor (c - \varepsilon ) F \rfloor$ is reduced and does not change as $\varepsilon \to 0$. The \emph{jumping divisor} $E^c$ of $c$ is 
	\begin{equation}\label{round}
		E^{c} =  \lfloor cF \rfloor - \lfloor (c - \varepsilon ) F \rfloor,
		\end{equation}
		where $\varepsilon$ is a sufficiently small positive number. 
	\end{dfn}
	
	The jumping divisor  $E^c$ is a reduced non-zero divisor whenever
	$c$ is a candidate  jumping number of $\fa$, by definition of candidate jumping number (see Section~\ref{candidate}).  Otherwise,  $E^{c}$ is zero.
	Further note that $E^{c} = E^{c+n}$ for any natural number $n$, as  follows from formula (\ref{round}) by properties of rounding down.

	\begin{rem}
	The notion of a jumping divisor was introduced as the \emph{maximal jumping divisor} in \cite{MR3510908} along with its \emph{minimal} variant in the two dimensional case. For simplicity, we  drop the adjective `maximal' from the name and do not discuss the minimal jumping divisor here.
	\end{rem}

	\begin{rem} \label{exceprem}
		Let $\mu: Y \to X$ be any log resolution of the closed subscheme $Z$ defined by $\fa$. Let   $\fa \cdot \cO_{Y} = \cO_{Y}(-F)$. Since we do not assume $X$ to be a projective variety (in fact, we will assume it to be affine), the divisor $F$ or its components are not necessarily projective varieties. However, we can realize each  prime component of the divisor $F$  as a projective variety over a suitable field, after a suitable base-change, as guaranteed by the next lemma.
		\end{rem}
		
		\begin{lem} \label{exceplem}
		With notation as in Remark~\ref{exceprem},
		\begin{enumerate}
		    \item The push-forward  $\mu_*\cO_{Y}(-F)$ is  the integral closure of the ideal $\fa$. 
		    \item If $E \subset Y $ is an irreducible component of $F$ then $\mu(E)$ is contained in $Z$.  Conversely, all prime divisors that are mapped inside $Z$ by $\mu$ occur as irreducible components of $F$.
		   \item If $Z$ is supported at a closed point $x$, then each component of $F$ is a smooth projective variety over $k$ and is contracted by $\mu$ to $x$.
		   
		   \item More generally, if $E \subset Y$ is a prime component of $F$ and $p \in Z$ is the center of $E$ on $X$ (i.e. generic point of $\mu(E)$), then
		$$ E^{\prime} = E \times_{X} \Spec(\cO_{X,p}) $$
		is a smooth projective variety over  the residue field $\kappa(p)$ of $p\in X$.
		
		\item The dimension of $E^{\prime}$ is $h-1$ where $h$ is the dimension of $\cO_{X,p}$. 
			\end{enumerate}
	\end{lem}
	\begin{proof}
	    \begin{enumerate}
	        \item See \cite[Remark 9.6.4]{MR2095472}.
	        \item The forward implication follows from part (1) and the converse from the fact that $\mu^{-1} (\fa) = \cO_{Y}(-F)$.
	        \item [(3)] This is a special case of part (4).
	        \item [(4)] Since we are assuming our log resolutions are projective maps, $E^{\prime}$ is projective over $\Spec(\cO_{X,p})$. Further, the fiber over $p$ is a projective variety over $\kappa(p)$. Hence, so is every closed subvariety of the fiber. Since $E^{\prime}$ is contracted to $x \in \Spec(\cO_{X,p}) $ by $\mu$, $E^{\prime}$ is a projective variety over $\kappa(p)$. Smoothness follows from the fact that $E$ is smooth over $k$.
	        
	       \item [(5)] Let $\kappa(E) = \kappa(E^{\prime})$ denote the function field (over $k$) of the variety $E$ and $\kappa(p)$ denote the residue field of $X$ at $p$. Then,
		$$ \dim_{\kappa(p)} E^{\prime} = \text{tr. deg.}_{\kappa(p)} \kappa(E^{\prime}) = \text{tr. deg.}_{k} \kappa(E^{\prime}) - \text{tr. deg.}_{k} \kappa(p) = d-1 - \dim_{k} (Z_{1}) = h-1$$
		where $Z_{1}$ denotes the closure of $p$ in $X$ and tr. deg.$_{\kappa}$ denotes the transcendence degree of a field over $\kappa$.\end{enumerate}
		\end{proof}
	
	\begin{notation} \label{localize}
	    For any effective divisor $D \subset F$ and a point $p$ in $X$, we denote by $D_{p}$ the scheme $D \times_{X} \Spec(\cO_{X,p})$ over $\Spec(\cO_{X,p})$. In particular $E^{c}_{x}$ denotes the base-change of the jumping divisor $E^{c}$ to $\Spec(\cO_{X,x})$. Note that $E^{c}_{x}$ can be empty. But in that case, it follows from equation \ref{secondexactseq} below that the multiplicity of $c$ at $x$ is zero, hence $c$ can not be a jumping number.
	\end{notation}
	\vspace{0.1in}
	
	We can now prove Theorem~\ref{thm3.1}.
	
	\begin{proof}[Proof of Theorem~\ref{thm3.1}]
		Since the multiplicities $m(c)$ depend only on the stalk of the multiplier ideal at $x$, we may assume $X$ is the local affine scheme $\Spec(\cO_{X,x})$.

	 Let $\mu: Y \to X$ be any log resolution of the closed subscheme $Z$ defined by the ideal $\fa$. Let $\fa \cdot \cO_{Y} = \cO_{Y}(-F)$. We now prove that for any $c> 0$, the multiplicity $m(c)$ can be calculated as an Euler characteristic on $Y$. To see this, we start with the exact sequence
		$$ \begin{tikzcd}
			0	 \arrow[r] & \cO_{Y}(-E^{c})  	\arrow[r] & \cO_{Y} 	\arrow[r] & \cO_{E^{c}} 	\arrow[r] & 0	
		\end{tikzcd} . $$
	Let $ \iL$ be the line bundle on $Y$ corresponding to the divisor $ K_{Y|X} +E^{c} - \lfloor cF \rfloor   = K_{Y|X} - \lfloor (c - \varepsilon)F \rfloor$. Tensoring with the invertible sheaf associated to $\iL$, we get 
		\begin{equation} \label{exactseq}
			\begin{tikzcd}
				0	 \arrow[r] & \cO_{Y}(\iL) \otimes \cO_{Y}(-E^{c})  	\arrow[r] & \cO_{Y} (\iL)	\arrow[r] & \cO_{E^{c}}(\iL|_{E^{c}}) 	\arrow[r] & 0	
			\end{tikzcd} .
		\end{equation}
		
		Now, observing that $\cO_{Y}(\iL) \otimes \cO_{Y}(-E^{c}) \isom  \cO_{Y}(K_{Y|X} - \lfloor cF \rfloor) $, the Local Vanishing Theorem (Theorem~\ref{localvanishing}) tells us that the sequence remains exact after applying $\mu_{*}$. Thus we have the exact sequence
		\begin{equation} \label{secondexactseq}
		    \begin{tikzcd}
			0 	\arrow[r] & \iI(\fa^{c}) \arrow[r] &  \iI(\fa^{c -\varepsilon}) 	\arrow[r] &  H^{0}(E^{c}, \iL|_{E^{c}})  \arrow[r] &  0 
		\end{tikzcd} ,
		\end{equation} 
		where the first map is just the inclusion of the ideal $\iI(\fa^{c}) $ inside $\iI(\fa^ {c - \varepsilon})$.
		
		
		Since $X$ is affine, applying the Local Vanishing Theorem (Theorem~\ref{localvanishing}) again, we have that the first two sheaves in the short exact sequence (\ref{exactseq}) have vanishing higher cohomology. Using the long exact sequence for (\ref{exactseq}) of sheaf cohomology groups, we get:
		\begin{equation} \label{vanishingcoh}
		    H^p(E^{c}, \iL|_{E^{c}}) = 0  \quad   \text{for all $p>0$.}
		\end{equation}
		
		 Since $E^{c}$ is reduced, by Lemma~\ref{exceplem}, it is a projective variety over $L$ (the residue field at $x$). Putting this together with equations (\ref{exactseq}) and (\ref{vanishingcoh}), we get
			\begin{equation} \label{eulerchar}
			m(c) =  	\lambda( \iI(\fa^{c -\varepsilon}) /\iI(\fa^{c}))  = \lambda(H^{0}(E^{c} , \iL|_{E^{c}}))	 = \dim_{L}(H^{0}(E^{c}, \iL|_{E^{c}})) = \chi_{L} (\iL |_{E^{c}}),
		\end{equation}  
where the third equality holds because the $\cO_{X,x}-$module $H^{0}(E^{c} , \iL|_{E^{c}})$ is already a vector space over $L$, since $E^{c}  $ is a projective variety over $L$.

		
		By equation (\ref{eulerchar}) above, we have 
		\begin{equation} \label{euler}
			f_{c}(n) =  \chi_{L} \Big( \cO_{Y}\big( K_{Y|X} - \lfloor (c- \varepsilon) F \rfloor \big) |_{E^{c}} \otimes \cO_{Y}\big( -nF \big) |_{E^{c}} \Big)
		\end{equation}
		for all $n \geq 0$. Since we are fixing $c$ and varying $n$, we may apply Theorem~\ref{Snapper} on the complete scheme $E^{c}$ over $L$ with $r = 1$, $\iF = \cO_{Y}\big( K_{Y|X} - \lfloor (c- \varepsilon) F \rfloor \big) |_{E^{c}}$ and $\iL_{1} = \cO_{Y}\big( -F \big) |_{E^{c}}$, to get a polynomial $Q_{c}(x)$ of degree at most the dimension of $E^{c}$ such that $Q_{c}(n) = \chi\Big( \iF \otimes \iL_{1}^{n} \Big) = f_{c}(n)$ for all $n  \geq 0 $. The proof is now complete by noting that the dimension of $E^{c}$ is equal to $h-1$ by Lemma~\ref{exceplem}, where $h$ is the codimension of $Z_{1}$.
		\end{proof}

	We can also explicitly determine the coefficient of $n^{h-1}$ in the polynomial $m(c+n)$:
	\begin{thm}
		\label{rhoc} Fix $c>0$.
		With notation as before, the coefficient $\rho_{c}$, of the term $n^{h-1}$ in the polynomial $f_{c} (n)$ describing the multiplicities $m(c+n)$ can be computed on the log resolution $Y$ using the formula:  
		\begin{equation}\label{FE}
		    \rho_{c} = \frac{(-1)^{h-1}}{(h-1)!} 	(\underbrace{ F|_{E^{c} _{x}} \cdot \ldots \cdot F|_{E^{c}_{x}}}_{h-1 \textrm{ times}}).
		\end{equation}
		where $E^{c}_{x} = E^{c} \times_{X} \Spec(\cO_{X,x})$ is a projective variety over $L$. In case the subscheme $Z$ is supported only at a closed point $x$ in $X$, the formula for $\rho_{c}$ may be written as follows (where $d$ is the dimension of $X$):
		\begin{equation} 
			\rho_{c} = \frac{(-1)^{d-1}}{(d-1)!} 	(\underbrace{ F \cdot \ldots \cdot F }_{d-1 \textrm{ times}} \cdot E^{c})
		\end{equation}   
		
	\end{thm}
	
	\begin{proof}
		Let $\iF_{c} = \cO_{Y}(K_{Y|X} - \lfloor (c- \varepsilon) F \rfloor)$ and $\iL = \cO_{Y}(-F)$. Then $f_{c}(n) = \chi(\iF|_{E^{c}_{x}} \otimes \iL|_{E^{c}_{x}} ^{n})$ by (\ref{euler}). Using part (\ref{intersection3}) of Proposition~\ref{interpn},  we have, $\rho_{c} =\frac{ (\iF_{c}|_{E^{c}_{x}}; \iL|_{E^{c}_{x}} \cdot \ldots \cdot \iL|_{E^{c}_{x}})}{(h-1)!}$.
		Since $\iF_{c}$ is a line bundle on $Y$ and by definition $\iL = \cO_{Y}(-F)$, using part (\ref{intersection4}) of Proposition~\ref{interpn}, we have:
		\begin{equation*} 
			\rho_{c} = (-1)^{h-1} \frac{	( F|_{E^{c} _{x}} \cdot \ldots \cdot F|_{E^{c} _{x}})}{(h-1)!} .
		\end{equation*}
		
		If $Z$ is supported only at a closed point $x$, then $E^{c} _{x} = E^{c}$ and is projective over $k$. By part (\ref{intersection1}) of Proposition~\ref{interpn}, we can compute $\rho_{c}$ on any projective closure of $Y$ as: $$\frac{ (\iL|_{E^{c}} \cdot \ldots \cdot \iL|_{E^{c}})}{(d-1)!} = \frac{ (\iL \cdot \ldots \cdot \iL \cdot E^{c})}{(d-1)!} = (-1)^{d-1} \frac{ ( F \cdot \ldots \cdot F \cdot E^{c})}{(d-1)!} .$$ 
			\end{proof}
	
	\begin{dfn} \label{rhocdef}
		Given $Z$ and $Z_{1}$, we define the \emph{Rees coefficient} of $c > 0 $ (denoted by $\rho_{c}$) to be the coefficient of the term $n^{h-1}$ in the polynomial $f_{c}(n)$. Equivalently, in light of Theorem~\ref{thm3.1}, $\rho_{c}$ is the unique number such that
		$$ m(c+n) = \rho_{c} n^{h-1} + o(n^{h-2}) \quad \text{as $n \to \infty$} 	,	$$ 
	 where $h$ is the codimension of $Z_{1}$ in $X$. \end{dfn}
	The Rees coefficient of a jumping number will be studied in more detail in the next section.
	
	\begin{rem}
		When $\dim (X) =2$ and $Z$ is a point scheme, the polynomial $f_{c}(n)$ has degree at most 1. In this case we have,
		$$ m(c+n) = m(c) + \rho_{c} n,		$$
		where $\rho_{c} = - F \cdot E^{c}$. This recovers the \emph{linear} formula for $m(c+n)$ proved in \cite{MR3558221}.
	\end{rem}
		
	\begin{rem}
		In higher dimensions and $Z$ still a point scheme, the polynomials $f_{c} (n)$ have other coefficients that can be computed as follows: \\
		If $f_{c}(n) = m(c) + \alpha_{1} ^{c} n + \dots + \alpha_{d-1} ^{c} n^{d-1}$, then
		\begin{equation}
			\alpha_{j} ^{c} = \frac{1}{j!} \int \limits _{E^{c}} c_{1}(\iL|_{E^{c}})^{j} \cap \tau_{E^{c}, j} (\iF_{c}|_{E^{c}})  ,
		\end{equation}
		where $\iF_{c} = \cO_{Y}(K_{Y|X} - \lfloor (c- \varepsilon) F \rfloor)$ and $\iL = \cO_{Y}(-F)$, $c_{1}$ denotes the first Chern class of a line bundle and $\tau _{E^{c},j}$ denotes the degree $j$ component of the Todd class of a sheaf. This formula comes from equation \ref{euler} and the Riemann-Roch Theorem for singular varieties. See \cite[Example 18.3.6]{MR1644323} for the details.
	\end{rem}
	
	The interpretation of the multiplicity $m(c)$ as the dimension of global sections as in (\ref{eulerchar}) also implies that the function under consideration, $f_{c}(n)$ is a non-decreasing function  which we now note:
	
	\begin{Pn} \label{incr}
		Let $Z$ be a closed subscheme of a smooth variety $X$ and $Z_{1}$ an irreducible component of $Z$ with generic point $x$. Fix any $c > 0$, then $m(c+1) \geq m(c)$.		
		
	\end{Pn} 
	
	\begin{proof}
		As before, we assume $X$ is the local scheme $\Spec(\cO_{X,x})$. It is sufficient to deal with the case when $c$ is a jumping number since otherwise $m(c) = 0$ and $m(c+1) \geq 0$ by definition. So we assume that $c$ is a jumping number, hence $m(c) > 0$. By (\ref{eulerchar}), we need to show that 
		$$\dim_{L} H^{0}\Big( E^{c}, \cO_{Y}(K_{Y|X} - \lfloor (c+1 - \varepsilon)F \rfloor)|_{E^{c}}\Big) \geq \dim_{L} H^{0} \Big(E^{c},  \cO_{Y}(K_{Y|X} - \lfloor (c - \varepsilon)F \rfloor)|_{E^{c}}\Big).$$
		Denoting by $\iF$ and $\iL$ the invertible sheaves $\cO_{Y}(K_{Y|X} - \lfloor (c - \varepsilon)F \rfloor)$ and $\cO_{Y}(-F)$ on $Y$ respectively, we need to prove:
		$$\dim_{L} H^{0}\Big( E^{c}, \iF|_{E^{c}} \otimes \iL|_{E^{c}} \Big) \geq \dim_{L} H^{0} \Big( E^{c}, \iF|_{E^{c}} \Big).$$
		But, since $\fa \cdot \cO_{Y} = \cO_{Y}(-F)$, $\iL$ is generated by its global sections (by the generators of $\fa$) on $Y$, the same is true for $\iL|_{E^{c}}$ on $E^{c}$. In particular, $\iL|_{E^{c}}$ has a non-zero global section. Hence, choosing a non-zero global section of $\iL|_{E^{c}}$, we have an injective map:
		$$ \cO_{E^{c}} \hookrightarrow \cO_{E^{c}}(\iL).		$$ 
		Tensoring with $\iF|_{E^{c}}$ preserves injectivity since $\iF$ was invertible. Hence, we have
		$$	\cO_{E^{c}}(\iF) \hookrightarrow \cO_{E^{c}}(\iL \otimes \iF).	$$
		Since taking global sections also preserves injectivity, our claim is proved.
	\end{proof}
	
	\begin{rem}
		This proposition implies that if $c$ is a jumping number of $Z$ at $Z_{1}$, then so is $c+1$. This also follows directly from the definition of multiplier ideals and jumping numbers.
		Though this proposition says that if we fix a `$c$' then $m(c+n)$ is non-decreasing with $n$, this does not imply that the multiplicity $m(c)$  increases with $c$ (i.e., as $c$ varies over the jumping numbers). In fact, this is not true. As we will see in the next section, $m(c+n)$ increases at different rates for different jumping numbers $c$.
	\end{rem}

		Theorem~\ref{thm3.1} can be thought of as a finiteness statement for multiplier ideals of a subscheme at its irreducible components.
		For example, we can use Theorem~\ref{thm3.1} to deduce a well-known periodicity result for jumping numbers:

	\begin{Pn} \label{periodic} Let $c$ be a positive real number, and let $Z_1$ be a codimension $h$ components of a subscheme $Z$ of a smooth variety $X$.
		If $c > h-1$, then $c$ is a jumping number of $Z$ at $Z_{1}$ if and only if $c+1$ is a jumping number of $Z$ at $Z_{1}$.
	\end{Pn}
	
	\begin{proof}
		By Proposition~\ref{incr}, if $c$ is a jumping number, then $c+1$ is also a jumping number (without any assumptions on $c$). So it is enough to prove that if $c+1$ is a jumping number of $Z$ at $Z_{1}$ then so is $c$. If $c$ is not a jumping number, then by Proposition~\ref{incr}, $m(c-n) = 0$ for all $n$ such that $ 0 \leq n \leq h-1$. By Theorem~\ref{thm3.1}, $m(c+n)$ is a polynomial of degree at most $h-1$. Since $m(c+n)$ has $h$ zeroes, this implies that $m(c+n)$ is identically zero. So $c+1$ can not be a jumping number either.
	\end{proof}

	\section{Properties of the Rees coefficient}
	In this section, we investigate some properties of the \emph{Rees coefficient} of a jumping number (Definition \ref{rhocdef}). In the first subsection \ref{positivity}, we prove some criteria for the positivity of the Rees coefficient. Next, we relate the Rees coefficients to the Hilbert-Samuel multiplicity (Section \ref{Hilbert}). We begin with some preliminary observations about the Rees coefficients.

	

 Recall that the Rees coefficient $\rho_c$  of $c$ is defined to be the coefficient of $n^{h-1}$ in the polynomial $m(c+n)$, where $m(c)$ denotes the multiplicity of $c$ of a closed subscheme $Z$ along one of its irreducible components $Z_{1}$ and $h$ denotes the codimension of $Z_{1}$ in $X$.	
	
	\begin{Pn} Let $X$ be a smooth variety and fix a closed subscheme $Z$ of $X$ and an irreducible component $Z_{1}$ of $Z$. Fix a positive real number $c$. Then the Rees coefficient satisfies the following:
	\begin{enumerate}
		\item $\rho_{c} \geq 0$ for all $c > 0$.
		\item $\rho_{c} \in \frac{1}{(h-1)!} \NN_{\geq 0}$.
	
		\item $\rho_{c} = \rho_{c+n}$ for any positive integer $n$.
			\item  $\rho_{c} = 0 $ if $c$ is not a candidate jumping number of $\fa$ (see Section~\ref{candidate}).
		\item	$\rho_{c}$ is positive implies that $c+n$ is a jumping number at $Z_1$ for some $n$.
	\end{enumerate}
	\end{Pn}
	\begin{proof}
	    By Theorem~\ref{thm3.1}, $m(c+n)$ is a polynomial in $n$ of degree at most $h-1$. Since $\rho_{c}$ is the coefficient of $n^{h-1}$, parts (1), (2) and (3) follow from the fact that $m(c+n)$ is a non-negative integer for all $n$. Parts (4) and (5) follow from the fact that $m(c+n)$ is positive exactly when $c+n$ is a jumping number. In particular, if $m(c+n)>0$ for some $n$, then $c$ (equivalently $c+n$) must be a candidate jumping number.
	\end{proof}
	Since $ \rho_{c} $ is the same as $\rho_{c+n}$ for any natural number $n$, we think of $\rho_{c}$ as a $\QQ$-valued function on $\QQ/\ZZ$. This function is an invariant of the subscheme $Z$ along a component $Z_{1}$ and describes interesting properties of jumping numbers at $Z_1$ (or rather their classes in $\QQ/\ZZ$). We turn to explaining these properties.

	\subsection{Criteria for Positivity} \label{positivity}
	We first discuss the issue of when the Rees coefficient $\rho_{c}$ is positive. More precisely, for any closed subscheme $Z$ and an irreducible component $Z_1$, we ask for which classes in $\QQ/\ZZ$ of jumping numbers $c$ is the Rees coefficient $\rho_{c}$ positive? Equivalently, for which jumping numbers $c$ at $Z_1$, does the multiplicity $m(c+n)$ along $Z_1$ grow like $n^{h-1}$ as $n \to \infty$ (where $h = \codim_{X} Z_{1}$)? In this subsection, we prove several criteria for when the Rees coefficient is positive. The following Theorem~\ref{mainthm41} is a summary of the results of this section: 
	
	\begin{thm} \label{mainthm41}
		Let $Z$ be a closed subscheme of $X$ (with ideal $\fa$) and $Z_{1}$ be an irreducible component of $Z$ with generic point $x$ in $X$. Let $h$ be the codimension of $Z_{1}$ in $X$ and fix any $c > 0$. Then the following are equivalent:
		\begin{enumerate}
			\item [(i)] The Rees coefficient $\rho_{c}$ is positive.
			\item [(ii)]  The polynomial $m(c+n)$ has degree $h-1$  in $n$.
			\item [(iii)] The jumping divisor $E^{c}$ (Definition~\ref{jumpingdiv}) contains a \emph{Rees valuation} of $\fa$ centered at $x$ on some (equivalently any) log resolution of $\fa$.
			\item [(iv)]  $c+h-1$ is a jumping number \emph{contributed by} some \emph{Rees valuation} of $\fa$ centered at $x$.
			\item [(v)] The class of $c$ in $\QQ/\ZZ$ is \emph{contributed by} some \emph{Rees valuation} of $\fa$ centered at $x$.
		\end{enumerate}

	\end{thm}
	
	We first explain the statement of the theorem and the various terms appearing in it before giving the proof.
	
	\medskip
	
	\subsubsection{\textbf{Jumping numbers contributed by divisors}} \label{contributedsection} The notion of a jumping number \emph{contributed by a divisor} on a log resolution was defined in \cite{MR2389246} by Smith and Thompson. This notion was studied further in \cite{MR2592954} and \cite{MR3737832}. This definition naturally generalizes to classes of jumping numbers in $\QQ/\ZZ$. In this terminology, Theorem~\ref{mainthm41} characterizes the classes \emph{contributed by} the Rees valuations on any log resolution as exactly the classes whose Rees coefficient is positive (hence the choice of name). We now review the definition of contribution by a divisor here.  Recall that $E^{c}_{x}$ is the base-change of the jumping divisor $E^{c}$ (Definition~\ref{jumpingdiv}) defined as $E^{c}_{x} = \big( \lfloor cF \rfloor - \lfloor (c - \varepsilon ) F \rfloor \big) \times_{X} \Spec(\cO_{X,x})$.
	\begin{dfn}\label{defncontri}
		
		Let $\mu: Y \to X$ be a log resolution of an ideal $\fa \subset \cO_{X}$ with $\fa \cdot \cO_{Y} = \cO_{Y}(-F)$ and $E \subset F$ be a prime  divisor. Then for any point $x$ of $X$, a jumping number $c$ of $\fa$ at $x$ is said to be \emph{contributed by $E$} if:
		\begin{enumerate}
		    \item $E_{x}$ is non-empty,
		    \item $E_{x} \subset E^{c}_{x}$ and
		    \item  $ \iI(\fa^{c})_{x} = \mu _{*} (\cO_{Y}(K_{Y|X} - \lfloor cF \rfloor))_{x} \subsetneqq \mu _{*} (\cO_{Y}(K_{Y|X} - \lfloor cF \rfloor + E))_{x} $.
		\end{enumerate} 
		
		We say that a class $[c] \in \QQ/\ZZ$ is \emph{contributed by $E$} if $c+n$ is contributed by $E$ for some natural number $n$.
	\end{dfn}
	The condition that $E_{x}$ is non-empty is saying that $E$ is centered on $\Spec(\cO_{X,x})$ and $E_{x} \subset E^{c}_{x}$ is just saying that $c$ is a candidate jumping number for $E$. We also note that contribution by a prime divisor depends only on the valuation defined by $E$ and not on the chosen log resolution. So we say that a jumping number is ``contributed by a valuation $\nu$" if it is contributed by a divisor on a log resolution whose associated valuation is $\nu$.
	
	\medskip
	
	\subsubsection{\textbf{Rees valuations:}} \label{reesvaldfn} Let $\fa \subset \cO_{X}$ be an ideal. If $\nu: \tilde{X} \to X$ denotes the normalization of the blow-up of $X$ along $\fa$, then $\fa \cdot \cO_{\tilde{X}} = \cO_{\tilde{X}}(-E)$, for an effective divisor $E$ on $\tilde{X}$. Then the valuations corresponding to the prime components of $E$ are called the \emph{Rees valuations} of $\fa$. 
	
	Let $(A,\fm)$ be a local ring and $\fa$ be an ideal. The set of Rees valuations of $\fa$ is a minimal set of discrete, rank one valuations $\nu_{1}, \dots , \nu_{r}$ of $A$  satisfying the property that:
	$$ \text{For all $n \in \NN$, the integral closure $\overline{\fa^{n} } = \bigcap_{i} \big( \fa^{n} V_{i} \cap A \big) $ where $V_{i}$  is the valuation ring of $\nu_{i}$}. $$ 
	In other words, the Rees valuations are a minimal set of valuations that determine the integral closure of all the powers of the ideal $\fa$.  We refer to \cite[Chapter 10]{MR2266432} for details about Rees valuations. This is also explained in \cite[Section 9.6]{MR2095472}.
	
	Given the normalized blowup $\tilde{X}$, or more generally any normal variety $Y$ mapping properly and birationally to $X$, we abuse terminology by saying ``$E$ is a Rees valuation of $\fa$" to refer to a prime divisor $E$ on $Y$ corresponding to a Rees valuation of $\fa$.

\medskip

	We now turn to the proof of Theorem~\ref{mainthm41}. 
	
	\medskip
	
	\paragraph{\textit{Proof of Theorem~\ref{mainthm41}:}}
	We first note the following easy implications in the theorem:
	The equivalence between (i) and (ii) is clear from Definition~\ref{rhocdef} of the Rees coefficient. That (iv) implies (iii) follows from Definition~\ref{defncontri} of jumping numbers contributed by divisors (so does (v) implies (iii)). We have that (iv) implies (v) immediately from Definition~\ref{defncontri} of classes of jumping numbers contributed by a divisor.
	So to complete the proof, we will prove that (i) is equivalent to (iii) and (iii) implies (iv). These will be proved as the two main theorems of this subsection.
	
	\medskip
	
	We first prove part (i) is equivalent to part (iii) of Theorem~\ref{mainthm41}:
	
	\begin{thm}\label{reesval}
		Let $Z$ be a closed subscheme of $X$ and $Z_{1}$ an irreducible component of $Z$ with generic point $x$ in $X$. Fix any $c > 0$. Then, the Rees coefficient $\rho_{c}$ at $x$ is positive if and only if on some (equivalently every) log resolution, the jumping divisor $E^{c}$ (Definition~\ref{jumpingdiv}) contains a Rees valuation of $Z$ centered at $x$.
		
	\end{thm}
	
	Let $\mu:  Y \to X$ be a log resolution of the ideal $\fa$ with $\fa \cdot \cO_{Y} = \cO_{Y}(-F)$.  When $X$ is a surface and $\fa$ is supported at a closed point, it follows from the results of Lipman, from the paper \cite{MR276239} that, for any irreducible exceptional curve $E \subset Y$, $(-F \cdot E) $ is non-negative and is positive exactly when $E$ corresponds to a Rees valuation of $\fa$. It follows from the formula \ref{FE} that the Rees coefficient $\rho_{c} $ is positive if and only if a Rees valuation appears in the jumping divisor $E^{c}$. Theorem~\ref{reesval} can be thought of as a generalization of Lipman's result in this context to higher dimensions. To prove this theorem, we need a key lemma which is a generalization to higher dimensions of a version of Lemma 21.2 from \cite{MR276239}.
	
	\begin{lem}\label{thm41}
		
		If $E \subset Y$ is an irreducible component of $F$ (and $x$ is the generic point of $Z_{1}$, an irreducible component of $Z$), then the following intersection number is non-negative:
		$$ (\underbrace{-F|_{E_{x}} \cdot \ldots \cdot -F|_{E_{x}}}_{h-1 \text{ times}}) \geq 0, $$
		where $E_{x} = E \times _{X} \Spec(\cO_{X,x})$ and $h$ is the codimension of $Z_{1}$ in $X$. Moreover, $(-F|_{E_{x}} \cdot \ldots \cdot -F|_{E_{x}})$ is positive if and only if the divisor $E$ is a Rees valuation of $\fa$ centered at $x$.
		
	\end{lem}
	
	\begin{proof}
		Since the relevant intersection numbers depend only on the local ring at $x$ in $X$, we assume $X$ is the local scheme $\Spec(\cO_{X,x})$. Recall that by Lemma~\ref{exceplem}, $E_{x}$ (henceforth denoted just by $E$) is a smooth projective variety over $L$ (the residue field of $X$ at $x$) of dimension $h-1$. So the intersection numbers make sense. 
		
		The set up of the proof is as follows: let $\tilde{\mu}: \tilde{X} \to X$ be the normalized blow-up of $\fa$, the ideal of $Z$ in $X$. Since $Y$ is normal (it is smooth) and $\fa  \cdot \cO_{Y}$ is locally principal, the universal property of normalization and of blowing-up provides a factorization 
		$$ \mu : Y \xrightarrow{\pi} \tilde{X} \xrightarrow{\tilde{\mu}} X		$$
		i.e., $\mu =  \tilde{\mu} \circ \pi$.
		Let $\iL$ denote the invertible sheaf $\cO_{Y}(-F) = \fa \cdot \cO_{Y}$ on $Y$ and $\tilde{\iL}$ be the invertible sheaf
		$ \fa \cdot \cO_{\tilde{X}}$ on $\tilde X$ and we have $\iL = \pi^{*}(\tilde{\iL})$. Then, since $\tilde{\iL}$ is relatively ample for $\tilde{\mu}$, if $D \subset \tilde{X}$ is any effective divisor, then $\tilde{\iL}|_{D}$ is ample.
		
		Let $B = \{G_{j}\}$ denote the finite set of prime divisors on $\tilde{X}$ in the support of $\fa \cdot \cO_{\tilde{X}}$ and $C = \{E_{i}\}$ denote the  finite set of irreducible components of $F$ on $Y$. Then, a divisor $E_{i}$ from the set $C$ corresponds to a Rees valuation of $\fa$ if and only if $\pi$ maps $E_{i}$ birationally onto some divisor $G_{j}$ from $B$. And if $E_{i}$ does not correspond to a Rees valuation, then $\pi$ maps $E_{i}$ onto a proper subset of $G_{j}$ from some $j$. In any case, for each $E_{i}$ in $B$, we have a $G_{j}$ in $C$ such that $\pi$ restricts to a map
		$$ \begin{tikzcd}
			E_{i} \arrow[d, hook] \arrow[r, "\pi|_{E_{i}}"] & G_{j} \arrow[d, hook] \\
			Y \arrow[r, "\pi"] & \tilde{X}  
		\end{tikzcd} $$ with $\pi$ birational exactly when $E_{i}$ corresponds to a Rees valuation of $\fa$. Now, we want to understand the intersection number $(\iL|_{E_{i}} \cdot \ldots \cdot \iL|_{E_{i}})$. By Proposition~\ref{interpn} (\ref{intersection5}), we have 
		\begin{equation} \label{projection}
		    (\underbrace{ \iL|_{E_{i}} \cdot \ldots \cdot \iL|_{E_{i}}}_{h-1 \text{ times}}) =  \text{deg}(\pi|_{E_{i}}) \times (\underbrace{\tilde{\iL}|_{G_{j}} \cdot \ldots \cdot \tilde{\iL}|_{G_{j}} }_{h-1 \text{ times}}). 
		    	\end{equation}
		    	
		    If $E_{i}$ is not centered at $x$ then $E_{i}$ does not appear on $Y$ since we are over $\Spec(\cO_{X,x})$ and if $E_{i}$ does not correspond to a Rees valuation then $\text{deg}(\pi|_{E_{i}})$ is zero. So, in both cases the intersection number  $(\iL|_{E_{i}} \cdot \ldots \cdot \iL|_{E_{i}})$ is zero. The only case when  $(\iL|_{E_{i}} \cdot \ldots \cdot \iL|_{E_{i}})$ can be non-zero is when $E_{i}$ is a Rees valuation centered at $x$. But in that case, $\pi|_{E_{i}}$ is birational onto $G_{j}$ and $\tilde{\iL}$ is ample on $G_{j}$. Hence by equation \ref{projection}, we get that $ (\iL|_{E_{i}} \cdot \ldots \cdot \iL|_{E_{i}}) $ is positive. This completes the proof of the lemma. 	\end{proof}

	\begin{proof}[Proof of Theorem~\ref{reesval}:]
	The theorem now follows from Lemma~\ref{thm41} and Theorem~\ref{rhoc}:
	$$ 	\rho_{c} =  (-1)^{h-1} \frac{	( F|_{E^{c} _{x}} \cdot \ldots \cdot F|_{E^{c}_{x}})}{(h-1)!} = \sum_{E \subset E^{c}} (-1)^{h-1} \frac{(F|_{E_{x}} \cdot \ldots \cdot F|_{E_{x}})}{(h-1)!}.	$$
	 where the second equality holds due to Proposition~\ref{interpn} (\ref{intersection4}). Lemma~\ref{thm41} tells us that each intersection number on the right-hand side is non-negative and is positive exactly when $E_{x}$ corresponds to a Rees valuation of $\fa$ centered at $x$.
	\end{proof}
	
	Finally, we conclude the proof of Theorem~\ref{mainthm41}, by proving that part (iii) implies (iv).
	
	\begin{thm} \label{contrithm}
		Let $Z$ be a closed subscheme of $X$ (with ideal $\fa$) and $Z_{1}$ an irreducible component of $Z$ with generic point $x$ in $X$. Fix a positive real number $c$. Suppose on some log resolution  $\mu: Y \to X$ of $\fa$, the jumping divisor $E^{c}$ (Definition~\ref{jumpingdiv}) contains a Rees valuation $E$ of $\fa$ centered at $x$, then $c+h-1$ is a jumping number at $x$ contributed by $E$ (where $h=\codim_{X} Z_{1}$).
	\end{thm}
	
	 The proof of this theorem has two main steps:
	\begin{itemize}
	    \item First, we reinterpret the notion of a jumping number \emph{contributed by a divisor $D$} in terms of non-vanishing of the space of global sections of a certain sheaf on the divisor $D$.
	    \item Next, we prove the non-vanishing of the required space of global sections by pushing the sheaf forward to the normalized blow-up of $\fa$ and using a theorem of Mumford's.
	\end{itemize}
	Both steps crucially rely on vanishing theorems, namely the Kawamata-Viehweg vanishing theorem and the Local vanishing theorem (Theorem~\ref{localvanishing}). We now recall the main theorems we need for the proof in the form that we will use them.
	
		\begin{thm}[Kawamata-Viehweg Vanishing Theorem, {\cite[Theorem 9.1.18]{MR2095472}}]   \label{KV vanishing}
		Let $X$ be a smooth projective variety of dimension $n$ and let $N$ be an integral divisor on $X$. Assume that
		$$	N \equiv_{num} B + \Delta $$
		where $B$ is a big and nef $\QQ-$divisor and $\Delta = \sum a_{i} \Delta_{i}$ is a $\QQ-$divisor with simple normal crossings support and such that $ 0 \leq a_{i} < 1$ for each $i$. Then,
		$$		H^{i} (X, \cO_{X} (K_{X} + N))= 0 \quad \text{for all } i > 0. $$
		\end{thm}
		
		\begin{thm}[Mumford's Theorem, {\cite[Theorem 1.8.5]{MR2095471}} ] \label{Mumford}
		    Let $X$ be a projective variety and $\cL$ a globally generated ample line bundle on $X$. Suppose a coherent sheaf $\cF$ on $X$ is $m-$regular with respect to $\cL$ i.e.
		    $$ H^{i} (X, \cF \otimes \cL^{m-i}) = 0 \quad \text{for $i>0$ }.      $$
		    Then, $\cF \otimes \cL^{m}$ is generated by its global sections.
		\end{thm}
	\medskip
	
	\begin{proof}[Proof of Theorem~\ref{contrithm}:]
	
	Contribution by divisors at $x$ depends only on the local ring at $x$ (See Definition~\ref{defncontri}). So we may replace $X$ by the affine local scheme $\Spec(\cO_{X,x})$. Further, since by assumption $ E \subset E^{c}$ and $E$ is centered at $x$, $E_{x} ^{c}$ contains $E_{x}$ (henceforth denoted by $E$), so $E^{c}_{x}$ is non-empty and $c$ is a candidate jumping number (defined in Section~\ref{candidate}).
	
	\medskip
	
	\paragraph{\textit{Step 1:}} Let $q > 0$ be a candidate jumping number and $G$ be a prime divisor on $Y$ with $G \subset E^{q}$. Then we claim that $q$ is a jumping number contributed by $G$ if and only if $ H^{0}(G, \cO_{G}(K_{G} - \lfloor qF \rfloor|_{G})) \neq 0 $. To see this, consider the following exact sequence on $Y$:
	$$ \begin{tikzcd}
		0	 \arrow[r] & \cO_{Y}(-G)  	\arrow[r] & \cO_{Y} 	\arrow[r] & \cO_{G} 	\arrow[r] & 0	
	\end{tikzcd}.  $$
	Tensoring with $\cO_{Y}(K_{Y|X} - \lfloor qF \rfloor + G)$,  we get
	$$	0	 \to \cO_{Y}(K_{Y|X} - \lfloor qF \rfloor)  	\to \cO_{Y} (K_{Y|X} - \lfloor qF \rfloor + G)
	\to \cO_{G}((K_{Y|X} - \lfloor qF \rfloor + G)|_{G}) 	\to 0 .$$
	
	Applying $\mu_{*}$, using the Local Vanishing Theorem (Theorem~\ref{localvanishing}) and the fact that $X$ is affine by assumption, we see that
	\begin{equation} \label{4.1}
	    \mu _{*} (\cO_{Y}(K_{Y|X} - \lfloor qF \rfloor + G))/\mu _{*} (\cO_{Y}(K_{Y|X} - \lfloor qF \rfloor)) \isom H^{0}(G, \cO_{G}(K_{Y|X} - \lfloor qF \rfloor + G)|_{G}).
	\end{equation}	   
	
	Further, since $G$ is a smooth divisor, we have
	\begin{equation} \label{4.2}
	     	\cO_{G}(K_{Y|X} - \lfloor qF \rfloor + G)|_{G} \isom \cO_{G}(K_{G} - \mu^{*} K_{X}|_{G} -\lfloor qF \rfloor|_{G}),
	\end{equation}	
	where we have used the adjunction formula for $G \subset Y$. Using equations \ref{4.1} and \ref{4.2}, we get:
	$q$ is a jumping number at $x$ contributed by $G$ (Definition~\ref{defncontri}) if and only if
	\begin{equation}\label{contricrit}
		H^{0}(G, \cO_{G}(K_{G} - \lfloor qF \rfloor|_{G})) \neq 0.
	\end{equation}
	Here we are using the following observation: since $G$ maps to $x$  (See Lemma~\ref{exceplem}), the canonical bundle $\omega_{X}$ pulls back to the trivial bundle on $G$. Note also that since $G$ is a smooth projective variety over $L$ (the residue field of $x$), $K_{G}$ denotes the canonical divisor of $G$ over $L$. This is justified by the following argument:
	
	Let $\varphi$ denote the structure map $G \to \Spec(L)$. We have the following exact sequence coming from the first exact sequence for differentials (\cite[Chapter II, 8.11]{MR0463157}):
	$$ \begin{tikzcd}
		0	 \arrow[r] & \varphi^{*}\Omega_{K|_{k}}  	\arrow[r] & \Omega_{G|_{k}} 	\arrow[r] & \Omega_{G|_{L}} 	\arrow[r] & 0	
	\end{tikzcd} . $$
	This sequence is exact on the left because  $\varphi^{*}\Omega_{L|_{k}}$ is a free sheaf on $G$ of rank equal to the transcendence degree of $L$ over $k$ ($= d-h$) and the sequence is exact at the generic point of $G$. Since $G$ is smooth over $L$, the other two sheaves are locally free of ranks $d-1$ and $h-1$. Taking top exterior powers and using the freeness of the first sheaf, we get the required isomorphism:
	$$ \omega_{G|_{k}} \isom \omega_{G|_{L}}  .   $$

	 This completes the \emph{Step 1} of the proof, where we have proved that the statement that $q+n$ is a jumping number contributed by $G$ for a natural number $n$, is equivalent to the statement that the invertible sheaf
	\begin{equation} \label{contriequation}
		\text{  $\cO_{G} (K_{G} - \lfloor qF \rfloor|_{G}) \otimes \cO_{G}(-nF|_{G})$ has a non-zero global section.}	
	\end{equation}
	
		\medskip
		
	\paragraph{\textit{Step 2:}} Now, we use the hypothesis that the divisor $E \subset E^{c}$ is a Rees valuation centered at $x$ to conclude that the sheaf $\cO_{E} (K_{E} - \lfloor cF \rfloor|_{E}) \otimes \cO_{E}(-(h-1)F|_{E}) $ has a non-zero global section.
	
	To do this, like in the proof of Lemma~\ref{thm41}, we consider the normalized blow-up $\tilde{X}$ of the ideal $\fa$ and by the universal property of blowing up and normalization, the map $\mu$ factors through $\tilde{X}$. The Rees valuation $E$ on $Y$ is the strict transform of a prime divisor $\tilde{E}$ on $\tilde{X}$ and $E$ maps birationally onto $\tilde{E}$ (since $\tilde{X}$ is normal and $\tilde{E}$ is a divisor). Note that $\tilde{E}$ is also a projective variety over $L$. This is summarized in the picture below:
	
	$$\begin{tikzcd}
		E	\arrow[d, hook] \arrow[r, "f|_{E}"] & \tilde{E} \arrow[d, hook] \arrow[r] &  \Spec(L)  \arrow[d, hook] \\	
		Y	 \arrow[r, "f"] & \tilde{X}  \arrow[r] & X
	\end{tikzcd}$$
	
	Let $\iL = \cO_{Y}(-F) = \fa \cdot \cO_{Y}$ and $\tilde{\iL} = \fa \cdot \cO_{\tilde{X}}$. Then, $\tilde{\iL}$ is very-ample on $\tilde{X}$ and $\iL = f^{*}(\tilde{\iL})$. 
	So, using the projection formula along $f|_{E}$, we have
	$$ H^{0} (E, \cO_{E} (K_{E} - \lfloor cF \rfloor|_{E} ) \otimes \iL^{n})   \isom H^{0}( \tilde{E},  f|_{E *}(\cO_{E}(K_{E} - \lfloor cF \rfloor|_{E})) \otimes \tilde{\iL} ^{n}	).$$
	
	So, it is enough to show that $ H^{0}( \tilde{E},  f|_{E *}(\cO_{E}(K_{E} - \lfloor c F \rfloor|_{E})) \otimes \tilde{\iL} ^{h-1}	) $ is non-zero. We do this by using Mumford's theorem (Theorem~\ref{Mumford}). And to check the required cohomology vanishing conditions, we use the Kawamata-Viehweg vanishing theorem (Theorem~\ref{KV vanishing}).
	
	Note that $\iL$ is a globally generated line bundle on $Y$. Hence, by Bertini's theorem (\cite[Chapter III, Corollary 10.9]{MR0463157}), we may choose a smooth divisor (possibly disconnected) $D$ linearly equivalent to $-F$. Further, we may also assume that $D$ intersects $E$ transversally (in particular, no component of $D$ is $E$). Having chosen such a $D$, we note that $- \lfloor cF \rfloor = \lceil -cF \rceil \sim_{\QQ} cD + \sum a_{i} E_{i}$ as $\QQ-$divisors with $0 \leq a_{i} < 1$ and where $E_{i}$ are the irreducible components of $F$ in some order. Let $\varDelta = \sum a_{i} E_{i}$, then since $E \subset E^{c}$, the coefficient of $E$ in $cF$ is already an integer, hence the coefficient of $E$ in $\varDelta$ is $0$. So, we have $- \lfloor cF \rfloor|_{E} \sim_{\QQ}   cD|_{E} + \varDelta|_{E} $. Since the support of $\varDelta$ is the union of divisors in the support of $F$ and does not contain $E$, it follows that $\varDelta|_{E} $ has simple normal crossings support (since $F$ was snc). Further, all the coefficients of $\varDelta|_{E}$ are in the interval $[0,1)$. Next, $\iL|_{E}$ is the pullback of $ \tilde{\iL}|_{\tilde{E}}$ along $f|_{E}$; since $f|_{E}: E \to \tilde{E}$ is a projective birational map, and $\tilde{\iL}$ is very ample on $\tilde{E}$, $\iL|_{E}$ is big and nef on $E$. So $cD|_{E}$ is a big and nef $\QQ$-divisor. Since $E$ is smooth and projective over $L$ and all the relevant hypotheses remain true after base changing to an algebraic closure of $L$, we can apply the Kawamata-Viehweg vanishing theorem (Theorem~\ref{KV vanishing}) on $E$ with $N = -\lfloor cF \rfloor|_{E} \equiv_{num} cD|_{E} + \varDelta|_{E}$ to conclude that $$	H^{i} (E, \cO_{E}(K_{E} - \lfloor cF \rfloor|_{E})) = 0 \quad	\text{for all $i > 0$.}	$$
	By the same argument, for all natural numbers $n \in \NN_{\geq 0}$, we have
	\begin{equation} \label{KV}
		H^{i} (E, \cO_{E}(K_{E} - \lfloor (c+n)F \rfloor|_{E})) = 0 	 \quad	\text{for all $i > 0$ and all $n \geq 0$.}	
	\end{equation}
	Now, \cite[Lemma 4.3.10]{MR2095471} implies that $ R^{j} f|_{E *}  \cO_{E}(K_{E} - \lfloor (c+n)F \rfloor|_{E})) = 0 $ for all $j > 0$ and all $n \geq 0$. By the Leray spectral sequence, we then have isomorphisms
	\begin{equation} \label{leray}	
		H^{i} (E, \cO_{E}(K_{E} - \lfloor (c+n)F \rfloor|_{E})) \isom H^{i}( \tilde{E},  f|_{E *}\cO_{E}(K_{E} - \lfloor (c+n)F \rfloor|_{E}))  
	\end{equation}
	for all $i \geq 0$ and all $n \geq 0$. So, if $\iF$ denotes the (coherent, since $f|_{E}$ is proper) sheaf $f|_{E*} \cO_{E}(K_{E} - \lfloor cF \rfloor|_{E})$ on $\tilde{E}$, then putting together (\ref{KV}) and (\ref{leray}), we get  
	\begin{equation} \label{CMregularity}
		H^{i}( \tilde{E},  \iF \otimes \tilde{\iL} ^{n} |_{\tilde{E}}) = 0 \quad \text{for all $i>0$ and $n \geq 0$.}
	\end{equation}
	Finally, (\ref{CMregularity}) implies that $\iF \otimes \tilde{\iL}^{h-1} |_{\tilde{E}}$ is $0-$regular with respect to the very ample line bundle $\tilde{\iL}|_{\tilde{E}}$ i.e.,
	
	$$	H^{i}( \tilde{E},  \iF \otimes \tilde{\iL} ^{h-1} |_{\tilde{E}} \otimes \tilde{\iL} ^{-i} |_{\tilde{E}}) = 0 \quad \text{for all $i>0$}	.$$
	
	This is because since the dimension of $\tilde{E}$ is $h-1$, we only need to check the vanishing of cohomology groups till $i=h-1$. But if $i \leq h-1$, then (\ref{CMregularity}) gives the required vanishing. So, by Mumford's Theorem (Theorem~\ref{Mumford}), we get that $\iF \otimes \tilde{\iL}^{h-1} |_{\tilde{E}}$ is globally generated, in particular, 
	$$	H^{0}( \tilde{E},  \iF \otimes \tilde{\iL} ^{h-1}|_{E} ) \neq 0	.$$
	So, using (\ref{contriequation}), we are done.
	This completes the proof of Theorem~\ref{contrithm} and hence of Theorem~\ref{mainthm41}.   
	\end{proof}
	
\medskip

	\begin{eg} \label{example} 	The Rees coefficient  $\rho_{c}$ (Definition~\ref{rhocdef}) of a jumping number $c$  is not always positive. Indeed, consider the ideal $\fa = (x^5 + y^3, y^4)$ in the polynomial ring $k[x,y]$ defining a point scheme supported at the origin in $X = \AA^{2}$. This ideal has Rees coefficient zero for all of its jumping numbers less than one. 
	
	To check this, observe that 	since dimension of $X$ is two, $\rho_{c}$ can be computed on any log resolution  $\mu: Y \to X$ of $\fa$, with $\fa \cdot \cO_{Y} = \cO_{Y}(-F)$ as $$ \rho_{c} = -F \cdot E^{c}, $$
	 where $E^{c}$ is the jumping divisor of $c$  as defined in Definition~\ref{jumpingdiv}. This follows by specializing Theorem \ref{rhoc} to the surface case, although it is also proved in \cite[Theorem~4.1]{MR3558221}.
	
	A log resolution $Y$ of $\fa$ can be computed by hand and requires $9$ successive blow-ups. Let $E_{1}, \dots, E_{9}$ denote the exceptional divisors obtained in the order they are labelled. Then, the intersection matrix for the resolution is:
		
		$$	\begin{pmatrix}
			-3 & 0 & 1 & 0 & 0 & 0 & 0 & 0 & 0 \\
			0  &-3 & 0 & 1 & 0 & 0 & 0 & 0 & 0 \\
			1  & 0 &-2 & 1 & 0 & 0 & 0 & 0 & 0 \\
			0  & 1 & 1 &-2 & 1 & 0 & 0 & 0 & 0 \\
			0  & 0 & 0 & 1 &-2 & 1 & 0 & 0 & 0 \\
			0  & 0 & 0 & 0 & 1 &-2 & 1 & 0 & 0 \\
			0  & 0 & 0 & 0 & 0 & 1 &-2 & 1 & 0 \\
			0  & 0 & 0 & 0 & 0 & 0 & 1 &-2 & 1 \\
			0  & 0 & 0 & 0 & 0 & 0 & 0 & 1 &-1 \\
		\end{pmatrix} $$
		The relative canonical divisor is $$K_{Y|X} = E_{1} + 2 E_{2} + 4 E_{3} + 7 E_{4} + 8 E_{5} +  9 E_{6} + 10 E_{7} + 11 E_{8} + 12 E_{9}.$$ If $\fa \cdot \cO_{Y} =  \cO_{Y}(-F) $, then $$F = 3 E_{1} + 5 E_{2} + 9 E_{3} + 15 E_{4} + 16 E_{5} + 17 E_{6} + 18 E_{7} + 19 E_{8} + 20 E_{9}.$$ Using this information, we  compute that the jumping numbers of $\fa$ less than one are
		$\frac{8}{15},  \frac{11}{15}, \frac{13}{15}$ and $\frac{14}{15}$--- for example,  by hand using the formula from \cite{MR3558221} or by using the 
	\emph{Macaulay2} package \cite{MultiplierIdealsDim2Source}. We can then compute that the jumping divisors for each of these jumping numbers turn out to be the same, namely 
	$E_{4}$. Therefore, using the intersection matrix above, we can compute that $-F \cdot E_c $ is zero, and conclude that the Rees coefficient for each of the jumping numbers 	$\frac{8}{15},  \frac{11}{15}, \frac{13}{15}$ and $\frac{14}{15}$ is zero in this case. \qed
	\end{eg}

		Theorem~\ref{mainthm41} and Example~\ref{example} suggest that the classes of jumping numbers in $\QQ / \ZZ$ naturally come in various types depending on the type of growth of the multiplicities of its translates. The type of jumping numbers for which the multiplicity grows fastest (i.e., for which the Rees coefficient is positive) is described by Theorem~\ref{mainthm41}. 
	The next  corollary states that the class of {\it integer} jumping numbers always has maximal growth of its multiplicities--that is $\rho_1>0$:

	\begin{Cor}	\label{notmaincor} Let $Z$ be a  closed subscheme of the smooth variety $X$, with  irreducible component $Z_{1}$  of codimension $h$.  Then we have,
	\begin{enumerate}
	    \item The Rees coefficient $\rho_{1} $ of the real number $1$ along $Z_{1}$ is always positive;
	    \item There are at most  $(h-1)! \, \rho_{1}$\ Rees valuations of $Z$ centered at $Z_1$.
	    \item The codimension $h$ is always a jumping number of $\fa$ contributed by each of the Rees valuations of $\fa$ centered at the generic point of $Z_{1}$.
	\end{enumerate}
	\end{Cor}
	
	Part (3) of Corollary \ref{notmaincor} recovers the  fact that, for a regular local ring $(A, \fm)$ of dimension $h$ essentially of finite type over $k$,  $h$ is always a jumping number of each $\fm-$primary ideal $\fa$.  
	
	\begin{proof}
		Consider a log resolution $\mu: Y \to X$ of $Z$ with $\fa \cdot \cO_{Y} = \cO_{Y}(-F)$.
		
		\begin{enumerate}
		    \item Since the jumping divisor $E^{1}$ (Definition~\ref{jumpingdiv}) is the same as $F_{\text{red}}$, it contains all the exceptional divisors corresponding to the Rees valuations. So, $E^{1}_{x}$ contains all the Rees valuations of $\fa$ centered at $x$. So, $\rho_{1}$ must be positive by Theorem~\ref{reesval}.
		    
		    \item By Theorem~\ref{rhoc}, we have
		    	$$ 	\rho_{c} =  (-1)^{h-1} \frac{	( F|_{E^{c} _{x}} \cdot \ldots \cdot F|_{E^{c}_{x}})}{(h-1)!} = \sum_{E \subset E^{c}} (-1)^{h-1} \frac{(F|_{E_{x}} \cdot \ldots \cdot F|_{E_{x}})}{(h-1)!}.	$$
	 where the second equality holds due to Proposition~\ref{interpn} (\ref{intersection4}). Since the intersection numbers $( -F|_{E_{x}} \cdot \ldots \cdot -F|_{E_{x}})$ are non-negative integers (by Lemma~\ref{thm41}) and are positive exactly for the Rees valuations at $x$, each such valuation adds at least $1/(h-1)!$ to $\rho_{1}$. Hence the claim.
		    \item The last statement follows immediately from Theorem~\ref{contrithm} since $E^{1}$ contains each of the Rees valuations at $x$.
		\end{enumerate}  
	\end{proof}

 For a closed subscheme $Z$ of $X$ with an irreducible component $Z_{1}$, Theorem~\ref{mainthm41} and the polynomial nature of the multiplicities $ m(c+n)$ guarantee more jumping numbers of $Z$ along $Z_{1}$ in the interval $(h-1, h]$ (where $h$ is the dimension of the local ring at $Z_1$). The following observations generalize similar ones made in \cite{MR3558221} when $\dim X=2$.
	
	\begin{Cor} \label{maincor} Fix an irreducible component $Z_{1}$ of a closed subscheme $Z$ of the smooth variety $X$.
		Let $\nu_{1}, \dots, \nu_{r}$ be all the Rees valuations of $Z$ (Definition~\ref{reesvaldfn}) centered at the generic point $x$ of $Z_{1}$. Set $a_{i}$ to be the order of vanishing of $Z$ along the valuation $\nu_{i}$, that is, set
		$$ a_{i} = \min \{ \nu_{i}(f) | f \in \fa \}, $$
		where $\fa$ is the ideal of $Z$. 
		Then we have,
		\begin{enumerate}
			\item Each rational number $\frac{\ell}{a_{i}}$ (where  $\ell \in \mathbb Z$) in the interval $ (h-1, h]$ is a jumping number at $x$ contributed by $\nu_{i}$ (in the sense of Definition~\ref{defncontri}).
			\item If $c$ is a jumping number at $x$ in the interval $(h-1, h]$ such that $c$ is not an integer translate of a smaller jumping number at $x$, then $c$ is a rational number of the form in part (1).
			\item More generally, if $c < h$ is a jumping number such that $c$ is not an integer translate of a smaller jumping number at $x$, then $m(c+n)$ is polynomial (in $n$) of degree at least  $\lfloor c \rfloor$.
		\end{enumerate}
	\end{Cor}
	
	\begin{proof} Clearly (3) implies (2). So, we just prove (1) and (3):
		\begin{enumerate}
			\item Let $\mu: Y \to X$ be a log resolution of $Z$ and let $E_{1}, \dots, E_{r}$ denote the exceptional divisors on $Y$ corresponding to $\nu_{1}, \dots, \nu_{r}$ respectively. We see that for any number of the form $\frac{\ell}{a_{i}} $, $E_{i}$ is contained in the corresponding jumping divisor (Definition~\ref{jumpingdiv}) and hence the Rees coefficient $\rho_{\frac{\ell}{a_{i}}} $ is positive by Theorem~\ref{reesval}. The claim now follows immediately from Theorem~\ref{contrithm}.
			
			\item [(3)] If $c$ is a jumping number such that $c-n$ is not a jumping number for any $n$, $m(\{c\} + n)$ is a polynomial (Theorem~\ref{thm3.1}) which is zero for $0 \leq n \leq  \lfloor c \rfloor -1$. Here $\{ c \}$ $(=  c - \lfloor c \rfloor)$ denotes the fractional part of the real number $c$. This means the degree must be at least $\lfloor c \rfloor$.
		\end{enumerate}
	\end{proof}
	
	\begin{rem} 
	Corollary~\ref{maincor} is interesting because it ensures that 
many jumping numbers in the interval $(h-1, h]$ can be computed easily from the normalized blow-up of a closed subscheme $Z$ (without computing a full log resolution of $Z$). Those that cannot are integer translates of a smaller jumping number.

In this context, we recall the result of Budur from \cite{MR2015069} mentioned in the introduction which implies that when $Z$ is a point scheme, the jumping numbers (along with their multiplicity) of $Z$ in the interval $(0,1)$ can be interpreted as coming from the cohomology of the Milnor fiber of a general element of $\fa$ (the ideal of $Z$). Putting this result together with Corollary~\ref{maincor} gives us an interpretation of essentially all jumping numbers of a point scheme in a smooth surface. In higher dimension, this only gives us an understanding of jumping numbers in the intervals $(0,1)$ and $(d-1, d]$. We raise the following two questions towards understanding all the other jumping numbers.
		
	\end{rem}
	
	\begin{ques}
		For each integer $j$ in the interval $ [0 , h-2]$, how do we characterise the classes of jumping numbers $[c]$ such that $m(c+n)$ grows like $n^{j}$? In particular, can we characterise such classes of jumping numbers in terms of contribution by divisors? 
	\end{ques}
	
	\begin{ques}
		Given a class of jumping numbers $[c] \in \QQ/\ZZ$ contributed by a Rees valuation, how do we find the smallest jumping number in $[c]$? More generally, how do we find the smallest jumping number in any given class of jumping numbers in $\QQ/\ZZ$? Equivalently, how do we characterize the jumping numbers $c \leq h-1$ of $\fa$ such that $c$ is not a translate of a smaller jumping number of $\fa$.
	\end{ques}
	
\subsubsection{\textbf{On the valuations contributing jumping numbers.}}	
	As another consequence of Theorem~\ref{contrithm}, we prove the following theorem:
		\begin{thm} \label{valuationthm}
	Let $X$ be a smooth variety over $k$. Let $\nu$ be a divisorial valuation over $X$ i.e., a divisorial valuation centered at a (not necessarily closed) point $x$ in $X$. Then there is an effective integral divisor $D$ on $X$ and $c > 0$, a jumping number of $D$ at $x$ such that $c$ is \emph{contributed by} $\nu$. 
		\end{thm}

	\begin{proof}
	 Let $(A, \fm)$ denote the local ring at $x$ and $h$ be the Krull dimension of $A$. First, suppose $h=1$ i.e., $\nu$ comes from a divisor $E$ on $X$. Then taking $D=E$ and $c=1$ works. So, we assume $h >1$.
	 
	 Next, we claim that there is an ideal $\fa$ in $A$ such that $\fa$ is $\fm-$primary and $\nu$ is a Rees valuation of $\fa$ (Defined in Section~\ref{reesvaldfn}). This follows immediately from  \cite[Proposition 10.4.4]{MR2266432}. So we may choose an ideal $\fa \subset \cO_{X}$ such that $x$ corresponds to a minimal component of $\fa$ and the stalk of $\fa$ at $x$ has $\nu$ as a Rees valuation. Then by Corollary~\ref{notmaincor}, we know that $h$ is a jumping number of $\fa$ at $x$ contributed by $\nu$. Now, choose a log resolution $\mu: Y \to X$ of $\fa$ and let $\fa \cdot \cO_{Y} = \cO_{Y} (-F)$, where $F= \sum a_{i} E_{i}$ and for prime components $E_{i}$. We may assume $\nu$ corresponds to $E_{1}$ on $Y$. Now, the fact that $h$ is contributed by $\nu$ means (by definition) that:
	\begin{equation} \label{contributioneq}
	    \iI(\fa^{h})_{x} = \mu _{*} (\cO_{Y}(K_{Y|X} -  hF ))_{x} \subsetneqq \mu _{*} (\cO_{Y}(K_{Y|X} -  hF  + E_{1}))_{x} .
	\end{equation}
Now, if we choose $D$ to be a general member of the ideal $\fa ^{n}$ (i.e., a general $k-$linear combination of its generators) for any $n > h$, then a general enough $D$ will be reduced (since $h>1$) and further, it will satisfy
$\mu^{*} D = \tilde{D} + nF$ where $\tilde{D}$ is the strict transform of $D$. Now, we claim that $\frac{h}{n}$ is a jumping number of $D$ at $x$ contributed by $E_{1}$. To verify this, we need to check that
    
       $$ \iI(\tfrac{h}{n} D)_{x} = \mu _{*} (\cO_{Y}(K_{Y|X} - \lfloor \tfrac{h}{n} (nF + \tilde{D})  \rfloor))_{x} \subsetneqq \mu _{*} (\cO_{Y}(K_{Y|X} - \lfloor \tfrac{h}{n} (nF+ \tilde{D}) \rfloor + E_{1}))_{x} $$
        which holds if and only if
       $$    \mu _{*} (\cO_{Y}(K_{Y|X} -  hF ))_{x} \subsetneqq \mu _{*} (\cO_{Y}(K_{Y|X} -  hF  + E_{1}))_{x} $$
which holds because $\frac{h}{n}$ is less than one and $\tilde{D}$ is reduced. So, we are done by (\ref{contributioneq}).
\end{proof}

\begin{rem}
    Theorem~\ref{valuationthm} shows that the set of valuations that contribute some jumping number of some divisor in $X$ includes all divisorial valuations over $X$, in contrast to the valuations computing only the log canonical threshold of divisors, which are known to satisfy many special properties (See \cite{MR4206085}). This negatively answers a question raised by Joaqu\'{i}n Moraga, asking whether any valuation contributing a jumping number also computes a log canonical threshold. We thank him for asking this question and useful related conversations.
\end{rem}

	\subsection{Relation to Hilbert-Samuel Multiplicity} \label{Hilbert}
	
	In this section, we relate the Hilbert-Samuel multiplicity of the subscheme $Z$ at an irreducible component $Z_{1}$ to the Rees coefficients $\rho_{c}$ associated to the jumping numbers of $\fa$. We first recall the relevant definitions.
	
	Let $(A, \fm)$ be the local ring of $X$ at the generic point of $Z_{1}$ and $\fa$ be the ideal of $Z$ in $A$. Let $h$ denote the Krull dimension of $A$. The Hilbert-Samuel multiplicity of $(A, \fm)$ with respect to the $\fm-$primary ideal $\fa$, $e_{\fa} (A)$ is defined to be the number 
	$$ 	e_{\fa}(A) =  \lim \limits_{n \to \infty} \frac{h! \, \lambda(A/\fa^{n})}{n^{h}}	$$
	where $\lambda$ denotes the length as an $A-$module. The number $e_{\fa}(A)$ is always a positive integer.
	
	 Recall that for any real number $c>0$, the Rees coefficient $\rho_{c}$ of $Z$ along $Z_{1}$ was defined in Definition~\ref{rhocdef} such that $m(c+n) = \rho_{c} n^{h-1} + o(n^{h-2})$ as $n \to \infty$. Here $m(c)$ denotes the multiplicity of the number $c$ (Definition~\ref{multidef}).

	\begin{thm} \label{HSmult}
		The following equality holds:
		$$ \sum \limits _{c \in (0,1]} \rho_{c} = \frac{e_{\fa}(A)}{(h-1)!}	.	$$
	\end{thm}
	
	The main idea relating these two numbers is that the Hilbert-Samuel multiplicity can be computed using multiplier ideals, which is a consequence of Skoda's Theorem: 
	\begin{lem}\label{mult} \label{lemmult}
		$$ e_{\fa}(A) = \lim \limits_{n \to \infty} \frac{h! \, \lambda(A/\iI(\fa^{n}))}{n^{h}} $$
		
	\end{lem} 
	
	\begin{proof}
		If $d$ is the dimension of the ambient variety $X$, for $ n \geq d$, we have
		$$ \fa^{n} \subset \iI(\fa^{n})  \subset \fa^{n-d+1}$$
		where the second containment holds because of Skoda's Theorem (Theorem~\ref{skoda}).
		This gives us
		$$ \lambda(A/\fa^{n-d+1}) \leq \lambda(A/(\iI(\fa^{n}))) \leq \lambda(A/\fa^{n}) .	$$
		When we divide by $n^{h}$ and take limit as $n \to \infty$, since both the first and third terms approach $e_{\fa}(A)/h!$, so does the middle term. 	\end{proof}
	
	\begin{proof}[Proof of Theorem~\ref{HSmult}]
	We first note that $\rho_{c}$ can be non-zero only for the candidate jumping numbers (Defined in Section~\ref{candidate}). Hence, the sum is finite. Now to prove the proposition, we use Lemma~\ref{lemmult} to compute $e_{\fa}(A)$ as follows:
	$$ \frac{e_{\fa}(A)}{h!} = 	\lim \limits_{n \to \infty} \frac{\lambda(A/\iI(\fa^{n}))}{n^{h}} = \lim \limits_{n \to \infty} \frac{\sum \limits_{c \leq n} m(c)}{n^{h}} =  \lim \limits_{n \to \infty} \frac{\sum \limits_{c \in (0,1]} \sum \limits _{j=0} ^{n-1} m(c+j)}{n^{h}}.	$$
	Since $m(c+j) = \rho_{c} j^{h-1} + o(j^{h-1})$ as $j \to \infty$, by using the fact that $\sum \limits_{j=0} ^{n} j^i = \frac{n^{i+1}}{i+1} + o({n^{i}})$ as $n \to \infty$, we have $ \sum \limits_{j=0} ^{n-1} m(c +j) = \frac{\rho_{c}}{h} n^{h} + o(n^{h-1}) $ as $n \to \infty$. Then, we have
	$$ 	 \frac{e_{\fa}(A)}{h!} = \lim \limits_{n \to \infty} \frac{\sum \limits_{c \in (0,1]} \frac{\rho_{c}}{h}n^{h} + o(n^{h-1})}{n^{h}} = \frac{\sum \limits _{c \in (0,1]} \rho_{c}}{h} 	$$which concludes the proof.
	\end{proof}
	
	\section{Special Case: Monomial Ideals}
	
	In this section, we derive an explicit formula for the multiplicities of jumping numbers and Rees coefficients (Theorem~\ref{monomial}) of an arbitrary cofinite monomial ideal (See Definitions~\ref{multidef} and \ref{rhocdef}). In particular, we see that the Rees coefficient is positive for the class of any jumping number in this case (Corollary~\ref{anotcor}).
	
	We introduce the relevant notation. Let $X = \AA^{d}$, $R = k[x_{1}, \dots , x_{d}]$ and $\fa$ be an $\fm-$primary monomial ideal where $\fm = (x_{1}, \dots, x_{d})$. 
	Fix the lattice $M = \NN^{d}$ inside $M_{\RR} = \RR^{d}$. For any $v = (v_{1}, \dots, v_{d}) \in M$, we write $x^{v} = x_{1}^{v_{1}} \cdots x_{d} ^{v_{d}}$ for the corresponding monomial. For any ideal generated by monomials $\fa$ in $R$, the \emph{Newton Polyhedron} $P(\fa) \subset \RR^{d}$ of $\fa$ is the convex hull of all vectors $v \in M$ such that $x^{v}$ belongs to $\fa$. Note that  $P(\fa)$ is an unbounded polyhedral (i.e., bounded by polygonal faces) region in the first orthant of $\RR^{d}$. For any positive real number $c$, let $P(c \cdot \fa)$ denote the polyhedron obtained by scaling the vectors in $P(\fa)$ by a factor of $c$. A theorem of Howald (\cite{MR1828466}) provides a formula for the multiplier ideals of $\fa$:
	
	\begin{thm}{(\cite[ Theorem 9.3.27]{MR2095472})}
	    Fix a monomial ideal $\fa$ in $R$ and a positive real number $c$. Then the multiplier ideal $\iI(\fa^{c})$ is the monomial ideal generated by all monomials $x^{v}$ such that 
		$$ 	v + \mathbf{1} \in \text{Int}( P(c \cdot \fa)),	$$
		where $\mathbf{1}$ is the vector $(1, \dots, 1)$ and  $\text{Int}(P(c \cdot \fa))$ denotes the interior of the scaled Newton polyhedron $ P(c \cdot \fa)$.
	\end{thm}

	\medskip We can now state our formula for the multiplicity and Rees coefficient of a jumping number.
	\begin{notation} 
	We will use the following notation throughout this section.
	\begin{itemize}
	    \item For any set $S \subset \RR^{d}$, $\# (S)$ denotes the number of lattice points in $S$.
	    \item If $P$ is a face of a $d-$dimensional polyhedron in $\RR^{d}$, then $\vol(P)$ denotes the $d-1$-volume of $P$.
	\end{itemize}
	 
	\end{notation}
		
	\begin{thm} \label{monomial} Fix an $\fm-$primary monomial ideal $\fa$ and a positive real number $c$.
		Then the multiplicity of $c$ is given by:
		\begin{equation} \label{monomialmult}
		m(c) = \# \Big\{ \bigcup _{i} cP_{i} ^{\circ} \cap M \Big \},\end{equation} 
		where 
		the $P_i$ are all the bounded faces of the Newton polyhedron $P(\fa)$ and $$ cP_{i}^{\circ}  :=  cP_{i} -\Big( \bigcup_{1 \leq j \leq d}  j^{th} \text{ coordinate hyperplane} \Big) = cP_{i} \cap (\mathbf{1} + M) .$$
		Moreover, the Rees coefficient $\rho_{c}$ is given by
		$$ \rho_{c} = \sum \limits _{c\cH_{i}  \cap M \neq \emptyset} \vol(P_{i}),$$
		where $ \mathcal H_i$ is the unique hyperplane containing the face $P_i$. 
	\end{thm}

	To prove Theorem~\ref{monomial}, we need the following Lemma:
	
	\begin{lem} \label{secobs} For each jumping number $c$ of the  monomial ideal $\fa$, the $k$-vector space $\iI(\fa^{c - \epsilon})/\iI(\fa^{c})$ is isomorphic to the vector space generated by the monomials in $\partial P(c \cdot \fa) \cap (\mathbf{1} + M)$, where $\partial  P(c \cdot \fa) $ is the boundary of the dilated (by $c$) Newton polyhedron $P(c \cdot \fa)$. \end{lem}
	
	\begin{proof}
		Howald's theorem implies that	a real number $c > 0$ is a jumping number of $\fa$ if and only if $cP_{i} ^{\circ} \cap M $ is non-empty for some $i$. Moreover, the monomials that are contained in $\iI(\fa^{c - \epsilon})$ but not in $\iI(\fa^{c})$ are exactly the monomials $x^{v}$ such that $v + \mathbf{1}$ is on the boundary of $ P(c \cdot \fa)$. Therefore, the $k-$vector space $\iI(\fa^{c - \epsilon})/\iI(\fa^{c})$ is isomorphic to the vector space generated by the monomials in $\partial P(c \cdot \fa) \cap (\mathbf{1} + M)$.
	\end{proof}
	

	\begin{proof}[Proof of Theorem~\ref{monomial}]	
	The condition that $\fa$ is an $\fm-$primary monomial ideal is equivalent to the condition that the region $P(\fa)$ intersects each coordinate axis. In this case, the boundary of $P(\fa)$ is the union of the bounded faces $P_{1}, \dots, P_{r}$ and the unbounded faces exactly along the coordinate hyperplanes. The faces $P_{1}, \dots, P_{r}$ are polytopes of dimension $d-1$ with vertices in $M$. Each polytope $P_{i}$ is contained in a unique hyperplane $\cH_{i} \subset \RR^{d}$. 
	
	For the first claim, we note that addition by $\textbf{1}$ gives a one-to-one correspondence between vectors in $M$ and the vectors in $M$ that do not lie on any of the coordinate hyperplanes. Hence, the vectors $v$ in $M$ such that $v + \textbf{1} \in \partial P(c \cdot \fa)$ are in one-to-one correspondence with the vectors in $\partial P(c \cdot \fa)$ that do not lie on any of the coordinate hyperplanes. But, since the components of the boundary that do not lie on the coordinate hyperplanes are exactly the bounded components $P_{1}, \dots, P_{r}$, this proves the first claim. So the formula for the multiplicity follows from this and Lemma~\ref{secobs} above.
	
	For $c > 0$, let $cP_{i}$ denote the polytope obtained by scaling each vector in $P_{i}$ by $c$ and let $c\cH_{i}$ be the corresponding scaled hyperplane.
	
	For the last claim, we use the following fact: Let $Q$ be a convex polyhedron of dimension $\ell$ in $\RR^{\ell}$. Then, $$ \vol(Q) = \lim \limits_{t \to \infty} \frac{\#(tQ \cap \ZZ^{\ell})}{t^{\ell}} .$$
	
	We actually need a slightly more general version of this formula: Suppose $\{ Q_{n}\}_{n \geq 1} $ is a sequence of polyhedra of dimension $\ell$ in $\RR^{\ell}$ such that $Q_{n}$ is some (possibly non-integer) translate of $c_{n} Q_{1}$ where $\{c_{n}\}$ is a sequence of real numbers such that $c_{n} \to \infty$ as $n \to \infty$. Then
	
	$$ \vol(Q_{1}) = \lim \limits_{n \to \infty} \frac{\# (Q_{n} \cap M)}{c_{n}^{\ell}}  .$$
	
	This statement actually follows immediately from the previous statement with the additional observation that when we translate a polyhedron by any vector, then the difference in the number of lattice points is bounded above by a fixed multiple of volume of the boundary of the polyhedron. Hence, the difference does not really matter for the limit.
	
	Suppose $c\cH_{i} \cap M \neq \emptyset$, then the number points of $(c+n) \cH_{i} \cap M$ in $P((c+n) \cdot \fa)$ is the same as the number of lattice points in a translate of $\frac{c+n}{c} (cP_{i})$ in $c\cH_{i}$. This is because, any lattice point on $P_{i}$ gives us a natural bijection between the lattice points in $c\cH_{i}$ and $ (c+n) \cH_{i}$ (by translating). Under this translation, we can identify $(c+n)P_{i}$ in $(c+n)\cH_{i}$  with a translate of $ \frac{c+n}{c} (c P_{i})$ in $c\cH_{i}$. Hence, using the formula above with $\ell = d-1$, $\RR^{d-1} = c\cH_{i}$, $Q_{n} = (c+n)P_{i}$ and $c_{n} = \frac{c+n}{c}$,  we have
	$$ 	\vol(P_{i}) = \frac{1}{c^{d-1}} \vol(cP_{i}) =  \lim\limits_{n \to \infty} \frac{\#((\frac{c+n}{c} P_{i}) \cap M)}{(\frac{c+n}{c} )^{d-1}}	.$$
	
	Since the number of lattice points in the intersection of $(c+n)P_{i}$ and the coordinate hyperplanes or the intersection of $(c+n)P_{i}$ and $(c+n)P_{j}$ for $j \neq i$ grows at a rate of at most $n^{d-2}$,  the formula for $\rho_{c}$ follows. \end{proof}
	

	It follows immediately from the theorem that:
	
	\begin{Cor} \label{anotcor}
		If $c$ is a jumping number of an $\fm-$primary monomial ideal $\fa \subset k[x_{1}, \dots, x_{d}]$, then the Rees coefficient $\rho_{c}$ is positive.
	\end{Cor}
	
	\section{Poincar\'{e} Series}
	
	Let $X$ be a smooth variety. For any closed subscheme $Z$ of $X$ and an irreducible component $Z_{1}$ of $Z$, we can assemble the multiplicities of jumping numbers (Definition~\ref{multidef}) into a generating function called the \emph{Poincar\'{e} Series} of $Z$ at $Z_{1}$. More precisely, this is the generating function defined by:
	$$  \phi_{Z, Z_{1}} (T) = \sum \limits _{c \in (0, \infty)}  m(c) T^{c} = \sum \limits _{c \in (0,1]} \sum \limits _{n =0} ^{\infty} m(c+n) T^{c+n}  .    $$ 
	
	This series is clearly only a countable sum, because the jumping numbers of $Z$ are a discrete subset of rational numbers. Moreover, there is an $ \ell \in \NN$ such that every jumping number has a denominator a factor of $\ell$ i.e., every jumping number is of the form $ \frac{n}{\ell}$ for $n \in \NN$. Then, having chosen such an $\ell$ and setting $z = T^{1/\ell} $, we see that $\phi_{Z, Z_{1}}(T)$ is actually a power series in $z$. We use the polynomial nature of $m(c+n)$ proved in Theorem~\ref{thm3.1} to prove that this power series is actually a rational function:
	
	\begin{thm} \label{poincare}
		$\phi_{Z, Z_{1}}(z)  $ is a rational function of $z$ where $z = T^{1/\ell}$ and $\ell$ is as above. In fact, we have the following formula:
		\begin{equation} \label{poincareformula}
			\phi_{Z, Z_{1}} (T) = \sum \limits _{c \in (0,1]} \Big( \frac{m(c)}{(1-T)} + \frac{ \gamma_{c,1}  T}{(1-T)^{2}} + \cdots + \frac{ \gamma_{c,h-2} T} {(1-T)^{h-1}} + \frac{ \rho_{c} (h-1)!  T}{(1-T)^{h}} \Big) T^{c}
		\end{equation}
		for some rational numbers $\gamma_{c,i}$, where $h$ denotes the codimension of $Z_{1}$ in $X$.
	\end{thm}
	
	\begin{proof}
		For any $c \in (0,1]$, by Theorem~\ref{thm3.1} we have $$ m(c+n) = \alpha_{0,c} + \alpha _{1,c} n + \cdots + \alpha _{h-1, c} n^{h-1}. $$ (Of course, $\alpha _{h-1,c} = \rho_{c}$ and $\alpha _{0,c} = m(c)$.) Now, the list of polynomials $p_{0}(n) = 1, p_{1}(n) = n, p_{2} (n) = \dbinom{n+1}{2}, \dots, p_{h-1}(n) = \dbinom{n+h-2}{h-1}$ has exactly one polynomial of each degree between $0$ and $h-1$. So, these polynomials form a basis over $\QQ$ of the space of rational polynomials of degree less than $h$. So, we can write the polynomial $m(c+n)$ in $n$ in terms of the $p(i)$'s as follows:
		\begin{equation}
			m(c+n) = \sum \limits _{i=0} ^{h-1} \gamma_{c, i} p_{i} (n) \quad \text{for some rational numbers } \gamma_{c,i}  .
		\end{equation}
		
		Inspecting the constant and the leading coefficient, we get $\gamma_{c,0} = m(c)$ and $\gamma _{c,h-1} = (h-1)! \times \rho_{c} $.
		Thus,
		$$
		\begin{aligned}
			\phi_{Z, Z_{1}} (T) &= \sum \limits _{c \in (0, \infty)}  m(c) T^{c} = \sum \limits _{c \in (0,1]} \sum \limits _{n =0} ^{\infty} m(c+n) T^{c+n} \\
			&= \sum \limits _{c \in (0,1]} T^{c} \sum \limits _{n =0} ^{\infty} ( \gamma_{c,0} p_{o}(n) + \gamma _{c,1} p_{1}(n) + \cdots + \gamma _{c, h-1} p_{h-1}(n) ) T^{n} \\ 
			&= \sum \limits _{c \in (0,1]} T^{c} \Big(\gamma_{c,0} \sum \limits _{n =0} ^{\infty}  p_{o}(n)  T^{n}  +  \gamma_{c,1}  \sum \limits _{n =0} ^{\infty} p_{1}(n) T^{n}  + \cdots + \gamma_{c, h-1} \sum \limits _{n =0} ^{\infty}  p_{h-1}(n) T^{n} \Big).
		\end{aligned}
		$$
		
		Now the theorem follows from the following elementary observation: $$ \sum _{n=0} ^{\infty} p_{i}(n) T^{n} = \begin{cases}
			\frac{1}{1-T} \quad \text{if } i=0, \\
			\frac{T}{(1-T)^{i+1}} \quad \text{if } i \geq 1.
		\end{cases}$$\\
		When $i =0$ the formula is clear. For $i \geq 1$, the formula is equivalent to
		$$ \frac{1}{(1-T)^{i+1}} = \sum _{j = 0} ^{\infty} \dbinom{j+i}{i} T^{j}     $$
		which is a well-known combinatorial identity counting degree $j$ monomials in $i+1$ variables.
	\end{proof}
	
	\begin{rem}
		Theorem~\ref{poincare} was proved independently by \`{A}lvarez Montaner and N\'{u}\~{n}ez-Betancourt in \cite{AMNB21} using different methods. They also prove an analogous theorem for Test Ideals, which are positive characteristic counterparts of multiplier ideals. We are grateful to them for conveying the explicit form for the Poincar\'e Series as in Equation \ref{poincareformula}.
	\end{rem}
	
	\begin{rem}
		The case when $Z$ is a point scheme and $X$ is a surface of Theorem~\ref{poincare} was proved by Alberich-Carrami\~{n}ana et al. in \cite{MR3558221}, generalizing the work of Galindo and Monserrat in the case of simple complete ideals in \cite{MR2671187}. 
	\end{rem}

	\printbibliography
\end{document}